\documentclass[a4paper, oneside,11pt]{amsart}
\usepackage[utf8]{inputenc}
\usepackage[english]{babel}

\usepackage[big]{layaureo}

\usepackage{amssymb, latexsym}

\usepackage[colorlinks=true, linkcolor=red, urlcolor=black, citecolor=blue]{hyperref}

\theoremstyle{definition}
\newtheorem{definition}{Definition}[section]
\newtheorem{example}[definition]{Example}
\newtheorem{remark}[definition]{Remark}

\theoremstyle{theorem}
\newtheorem{theorem}[definition]{Theorem}
\newtheorem{lemma}[definition]{Lemma}
\newtheorem{proposition}[definition]{Proposition}
\newtheorem{corollary}[definition]{Corollary}

\DeclareMathOperator{\id}{\ensuremath{id}}
\DeclareMathOperator{\tr}{\ensuremath{tr}}

\DeclareMathOperator{\End}{\ensuremath{End}}
\DeclareMathOperator{\ad}{\ensuremath{ad}}
\DeclareMathOperator{\Ad}{\ensuremath{Ad}}
\DeclareMathOperator{\Rm}{\ensuremath{R}}
\DeclareMathOperator{\Ric}{\ensuremath{Ric}}
\DeclareMathOperator{\SU}{\ensuremath{SU}}
\DeclareMathOperator{\U}{\ensuremath{U}}
\DeclareMathOperator{\diag}{\ensuremath{diag}}
\DeclareMathOperator{\Ch}{\ensuremath{C}}
\DeclareMathOperator{\Bis}{\ensuremath{B}}

\DeclareMathOperator{\Tor}{\ensuremath{T}}
\DeclareMathOperator{\frakg}{\ensuremath{\mathfrak{g}}} 
\DeclareMathOperator{\frakh}{\ensuremath{\mathfrak{h}}} 
\DeclareMathOperator{\frakk}{\ensuremath{\mathfrak{k}}} 
\DeclareMathOperator{\frakl}{\ensuremath{\mathfrak{l}}} 
\DeclareMathOperator{\frakm}{\ensuremath{\mathfrak{m}}} 
\DeclareMathOperator{\frakn}{\ensuremath{\mathfrak{n}}} 
\DeclareMathOperator{\frakb}{\ensuremath{\mathfrak{b}}} 
\DeclareMathOperator{\frakd}{\ensuremath{\mathfrak{d}}} 
\DeclareMathOperator{\frakf}{\ensuremath{\mathfrak{f}}} 
\DeclareMathOperator{\fraks}{\ensuremath{\mathfrak{s}}} 
\DeclareMathOperator{\frakt}{\ensuremath{\mathfrak{t}}} 
\DeclareMathOperator{\fraku}{\ensuremath{\mathfrak{u}}} 
\DeclareMathOperator{\frakv}{\ensuremath{\mathfrak{v}}} 
\DeclareMathOperator{\frakz}{\ensuremath{\mathfrak{z}}} 
\DeclareMathOperator{\frsu}{\ensuremath{\mathfrak{su}}} 
\DeclareMathOperator{\cplx}{\ensuremath{c}}

\newcommand{\CC}{\ensuremath{\mathbb{C}}} 
\newcommand{\RR}{\ensuremath{\mathbb{R}}} 
\newcommand{\HH}{\ensuremath{\mathbb{H}}} 
\newcommand{\calB}{\ensuremath{\mathcal{B}}} 
\newcommand{\vareps}{\ensuremath{\varepsilon}} 

\providecommand{\abs}[1]{\lvert#1\rvert}
\providecommand{\dd}[1]{\ensuremath{\operatorname{d}\!{#1}}}

\begin{document}

\title{On homogeneous HKT manifolds and the Einstein condition}

\author{Lucio Bedulli and Lorenzo Marcocci}
\thanks{This work was partially supported by GNSAGA of INdAM and by the Project PRIN 2022 CUP: 20225J97H5 ``Differential-geometric aspects of manifolds via Global Analysis''}
\address{Dipartimento di Ingegneria e Scienze dell'Informazione e Matematica, Università dell'Aquila, Via Vetoio, 67100 L'Aquila, Italy}
\email{lucio.bedulli@univaq.it}
\email{lorenzo.marcocci@graduate.univaq.it}

\date{}

\maketitle

\begin{abstract}
We consider homogeneous hypercomplex manifolds with a transitive action of a compact Lie group and we give a characterization of invariant HKT metrics on them.
On every such hypercomplex manifold we prove the existence of an invariant HKT-Einstein metric, which is unique up to scaling.
Furthermore,  we determine for which invariant HKT metrics the torsion and the curvature of the Bismut connection are Bismut-parallel, showing that invariant strong HKT metrics have this property.
\end{abstract}

\section{Introduction}\label{S:Intro}
The geometry of complex homogeneous manifolds is a traditional and active research field, which has its roots in the works by Samelson~\cite{Sam} and Wang~\cite{Wang} about the existence of complex structures on compact semisimple Lie groups and their coset spaces. 
There is a vast literature dealing with Kähler metrics in this setting, see e.g.~\cite{AP, PavingTheWay}, and more recently also Hermitian non-Kähler ones, including~\cite{Barbaro, FinoGrantcharov, LauretMontedoro, Podesta, PodestaZheng}.
In a similar vein, the aim of the present paper is to investigate hyperkähler metrics with torsion, in short \emph{HKT metrics}, on manifolds acted transitively by compact Lie groups, focusing on ``canonical'' ones.

HKT structures have been firstly introduced in theoretical physics by Howe and Papadopoulos (\cite{HowePapadopoulos}), together with the subclass of \emph{strong} HKT structures, having a significant role in certain supersymmetric sigma models.
Shortly after, starting with the paper by Grantcharov and Poon~\cite{GrantcharovPoon}, HKT structures attracted a great deal of attention for their geometric properties.
In some ways, they can be viewed as the analogue of Kähler metrics in hypercomplex geometry, 
in this direction see for example~\cite{BS, GLV}, and at the same time they generalize the notion of hyperkähler metrics.

In analogy with Kähler metrics, an important part of the research is the quest of canonical HKT metrics. Recently, an Einstein-type condition for HKT metrics has been introduced by Fusi and Gentili in~\cite{FusiGentili}.
The definition is  inspired by that of Kähler-Einstein metrics or, more generally, first Chern-Einstein metrics (see section~\ref{SS:HKT} for details).
Such \emph{HKT-Einstein} metrics are also stationary solutions of the geometric flow on hypercomplex manifolds considered in~\cite{BGV}.
Several properties of HKT-Einstein metrics are investigated in~\cite{FusiGentili}, as part of a broad study of special hyperhermitian metrics.\\

One of the remarkable features of HKT metrics is that they exist in abundance on compact homogeneous manifolds.
Indeed, among the first examples of (strong) HKT metrics, there are bi-invariant hyperhermitian metrics constructed in~\cite{GrantcharovPoon} on certain compact Lie groups with invariant hypercomplex structures (see also~\cite{OpfermannPapadopoulos}), and on the same group manifolds left-invariant  HKT-Einstein metrics have been found in~\cite{FusiGentili}.
Such Lie groups belong to the larger class of so-called \emph{Joyce hypercomplex manifolds} (\cite{Joyce}). 
They are homogeneous spaces of the form $G/L$, where $G$ is a compact Lie group and $L$ is a closed subgroup satisfying appropriate conditions, equipped with invariant hypercomplex structures.
The construction by Joyce is based on the choice of a maximal torus in $G$ and a maximal set of strongly orthogonal roots in the corresponding root system, which induce a decomposition of  the Lie algebra of $G$ of the form 
\begin{equation}\label{E:decomp}
	\frakg = \frakl \oplus \frakm, \quad\quad
	\frakm = \bigoplus_{j=1}^m \frakm_j,
\end{equation}
where $\frakm$ is an $\Ad(L)$-module and $\frakm_j$ are specific submodules.
This allows one to define $\Ad(L)$-invariant hypercomplex structures on each ``layer'' $\frakm_j$ independently, and in turn $G$-invariant hypercomplex structures on $G/L$, provided that $L$ is connected.
In fact, such structures form an $m^2$-parameter family. 
It has been proven in~\cite{DimitrovTsanovI, DimitrovTsanovII} that every invariant hypercomplex structure on a coset space $G/L$ as above can be obtained via the Joyce construction.
This had been shown previously in~\cite{BedulliGoriPodesta} assuming in addition the existence of a hyperhermitian naturally reductive metric, which is in particular HKT.\\

In this paper we study in more depth invariant HKT metrics on general Joyce hypercomplex manifolds.
In particular, we are able to characterize them in the following way.
\begin{theorem}\label{T:HKT}
Let $(M=G/L, I, J)$ be a Joyce hypercomplex manifold, with a fixed decomposition of $\frakg$ as in~\eqref{E:decomp}.
Then there is a hyperhermitian metric $h$ on $\frakm$, for which the layers $\frakm_j$ are orthogonal, such that every $G$-invariant HKT metric $g$ on $M$ corresponds to a scalar product of the form
\begin{equation}\label{E:scalarproduct}
	g = \sum_{j=1}^m g_j h|_{\frakm_j \times \frakm_j},
\end{equation}
for some positive real numbers $g_1, \dots, g_m$.
Conversely, if $L$ is connected, a scalar product of this form gives rise to an invariant HKT metric on $M$.
\end{theorem}
The hyperhermitian metric $h$ is obtained by suitably modifying the opposite of the Killing form.
It should be noted that we are not putting restrictions on the hypercomplex structures on each coset space, and the HKT metrics we find, in general, are not naturally reductive.\\


Using also this characterization, we can prove existence and uniqueness of invariant HKT-Einstein metrics on each Joyce hypercomplex manifold.
\begin{theorem}\label{T:HKT-E}
Let $(G/L, I, J)$ be a Joyce hypercomplex manifold. Then there exists a unique (up to scaling) $G$-invariant HKT-Einstein metric on $M$.
\end{theorem}
This should be regarded as the analogue of the well known result of existence and uniqueness up to scaling of invariant Kähler-Einstein metrics on homogeneous simply connected compact Kähler manifolds.\\

The HKT-Einstein condition is based on the curvature of the Chern connection.
On the other hand, on a HKT manifold there is also a natural hyperhermitian connection.
Indeed, one of the equivalent definitions of a HKT structure is that of a hyperhermitian structure $(I,J,g)$ whose Bismut connections all coincide.
Therefore is natural to investigate torsion and curvature of the unique Bismut connection.

The Bismut connection of homogeneous Hermitian manifolds associated to compact Lie groups has been investigated in several recent works.
In particular, in~\cite{FinoGrantcharov, LauretMontedoro}  it is proved, among other things, that a left-invariant Hermitian metric on a compact semisimple Lie groups is simultaneously CYT and SKT
if and only if it is bi-invariant. Related results concerning C-spaces have been found in~\cite{Barbaro}.

On the other hand, Hermitian metrics with parallel Bismut torsion, also called Bismut-torsion-parallel (BTP), have attracted attention in their own right.
In~\cite{PodestaZheng}, a characterization of BTP Hermitian metrics is given for certain classes of flag manifolds and for compact simple Lie groups. It is shown moreover that also the Bismut curvature is parallel, obtaining  examples of Bismut-Ambrose-Singer (BAS) manifolds.

In the case of Joyce hypercomplex manifolds, we obtain the following.

\begin{theorem}\label{T:Bismut}
Let $(M=G/L, I, J)$ be a Joyce hypercomplex manifold, with an invariant HKT metric  $g$.
\begin{enumerate}
\item
The metric $g$  is BTP if and only if the metric $\tilde{g}$ on the base of the Tits fibration $G/L\to F$ relative to $I$
is induced by the restriction of a $\Ad(G)$-invariant scalar product on $\frakg$, and in this case it is BAS.
\item
If $g$ is strong HKT, then it is BTP. 
\item
Assume $g$ is strong HKT and moreover $M=\Tor^{\ell} \times (K_1/L_1) \times \cdots \times (K_t/L_t)$, where $\Tor^{\ell}=\U(1)^{\ell}$ is a $\ell$-dimensional torus, $K_j$ are compact, connected, simple Lie groups, and $L_j\subset K_j$ are closed, connected subgroups. Then $M$ is a Lie group and the HKT structure  $(I,J,g)$ is left-invariant.
\end{enumerate}
\end{theorem}

Combining (3) with~\cite[Corollary 5.12]{DimitrovTsanovII}, we obtain that the only Joyce hypercomplex manifolds, which are simply connected and admit an invariant strong HKT metric, are  products $\SU(2n_1+1) \times \cdots \times \SU(2n_r+1)$ with a bi-invariant hyperhermitian metric.
Part (3) of Theorem~\ref{T:Bismut} should be compared with very similar results in~\cite[Theorems 3.5, 3.6]{BFGV}.\\

The paper is organized as follows.
In Section~\ref{S:Prelim} we collect preliminar material that will be used throughout the paper.
Section~\ref{S:HKT} is devoted to the characterization of invariant HKT metrics, that is Theorem~\ref{T:HKT}.
Section~\ref{S:HKT-Einstein} deals with HKT-Einstein metrics  and with Theorem~\ref{T:HKT-E}. Section~\ref{S:Bismut} contains some observations on the Bismut connection and the proof of Theorem~\ref{T:Bismut}. 

\subsection*{	Acknowledgements}
The authors would like to thank Elia Fusi, Giovanni Gentili and Fabio Podestà for useful discussions and remarks.

\section{Preliminaries}\label{S:Prelim}

\subsection{HKT structures}\label{SS:HKT}
A \emph{hypercomplex manifold} is a $4n$-dimensional smooth manifold $M$ equipped with two anti-commuting complex structures $I$ and $J$, which produce a 2-sphere of complex structures $\mathsf{H} := \{ aI + bJ + cK : a^2 + b^2 + c^2 = 1 \}$, where $K = IJ$.
We say that a Riemannian metric $g$ on $M$ is \emph{hyperhermitian} if it is Hermitian with respect to both $I$ and $J$, hence with respect to every complex structure $P$ in the 2-sphere.
We have the usual associated 2-forms $\omega_P := g(P \cdot, \cdot)$, and it is easy to see that $g$ is hyperhermitian if and only if $g$ is $I$-Hermitian and $J\omega_{I} = - \omega_{I}$.
In this case $g$ also defines a differential form of type $(2,0)$ with respect to $I$
\[
	\Omega := \frac{1}{2}\left( \omega_J + i \omega_K \right).
\]
One can also recover $g$ from $\Omega$ by the relation $g = 2 \operatorname{Re}( \Omega( \cdot, J\cdot) )$.

A hyperhermitian metric $g$ is called \emph{hyperkähler with torsion} (HKT, for short) if the Bismut connections of the Hermitian structures $(P, g)$ coincide for all $P\in\mathsf{H}$; this unique connection will be denoted by $\nabla^{\Bis}$.
A key property, proved in~\cite{GrantcharovPoon}, is that a hyperhermitian metric $g$ is HKT if and only if
\[
	\partial\Omega=0,
\]
where the operator $\partial$ is taken with respect to $I$.
Since $I,J,K$ and $g$ are parallel tensors with respect to $\nabla^{\Bis}$, the restricted holonomy group of $\nabla^{\Bis}$ is contained in $\operatorname{Sp}(n)$.
It then follows that each Hermitian structure $(P,g)$ of a HKT manifold is \emph{Calabi-Yau with torsion} (CYT), that is,
$\operatorname{Hol^{\circ}}(\nabla^{\Bis})\subset\SU(2n)$. 
This property is equivalent to the vanishing of the associated Bismut-Ricci form.

A HKT structure is called $\emph{strong}$ if one (and hence every) Hermitian structure $(P,g)$ is \emph{strong Kähler with torsion} (SKT).
By definition this means that the torsion 3-form of the Bismut connection, namely
$- \dd{}^{\cplx}_{I}\omega_I$, where $\dd{}^{\cplx} = I^{-1} \dd{} I$, is $\dd{}$-closed.

Consider now the Chern connection, denoted by $\nabla^{\Ch}$, of the Hermitian structure $(I, g)$, and let $R^{\Ch}(X,Y) = [\nabla^{\Ch}_{X}, \nabla^{\Ch}_{Y}] - \nabla^{\Ch}_{[X,Y]}$ be the curvature tensor.
The (first) \emph{Chern-Ricci form} $\Ric_{\omega_I}$ is the $(1,1)$-form with respect to $I$ defined by
\[
	\Ric_{\omega_I}(X,Y) 
	= \sum_{k=1}^{2n} 
		g\left( R^{\Ch}(X,Y) I e_k, e_k \right)
	= i \sum_{k=1}^{2n}
		g\left( R^{\Ch}(X,Y) v_k, \bar{v}_k \right),
\]
where $\{ e_k, Ie_k \}_{k=1}^{2n}$ is a local orthonormal frame, and $v_k = \frac{1}{\sqrt{2}}(e_k - iIe_k)$.
\begin{definition}[\cite{FusiGentili}]
A HKT metric $g$ is called \emph{HKT-Einstein} if there exists a function $\lambda\in \operatorname{C^{\infty}}(M; \RR)$ such that
\begin{equation}\label{E:HKT-E}
	\frac{\Ric_{\omega_I} - J\Ric_{\omega_I}}{2} = \lambda \omega_I.
\end{equation}
\end{definition}
Note that the left-hand side of equation~\eqref{E:HKT-E} is the $J$-anti-invariant part of $\Ric_{\omega_I}$.
It is proved in~\cite{FusiGentili} that the condition does not depend on the pair of anti-commuting complex structures chosen in $\mathsf{H}$, and that it can be rephrased as
\[
	\partial_{J}(\theta^{1,0}) = \lambda \Omega,
\]
where $\theta = - I\!\dd{}^*\omega_I$ is the Lee form of $(I, g)$, $\theta^{1,0}$ is its $(1,0)$-part with respect to $I$, and $\partial_{J} = J^{-1}\bar{\partial}J$.
Moreover, when $M$ is compact, $\lambda$ must be a non-negative constant, with value
\[
	\lambda = \frac{1}{2n\operatorname{Vol}(M,g)} 
				\int_M \abs{\theta}^2 \frac{\Omega^n\wedge\bar{\Omega}^n}{(n!)^2}.
\]
Since $\Ric_{\omega_I}$ is invariant by scaling of the metric, when searching for HKT-Einstein metrics it is not restrictive to assume $\lambda \in \{ 0, 1 \}$.

By the above formula, $\lambda=0$  if and only if $\theta=0$, i.e., $(I, g)$ is balanced.
Since for every HKT metric one has $\partial\bar{\Omega}^n = \theta^{1,0}\wedge\bar{\Omega}^n$, this in turn implies that the canonical bundle $K_{M,I}$ is holomorphically trivial.
Consequently, if $(M, I, J)$ is compact with $K_{M,I}$ not holomorphically trivial, it may carry only HKT-Einstein metrics with positive $\lambda$.

\subsection{Homogeneous Hermitian manifolds}\label{SS:homogeneous}
We recall here basic notions on homogeneous Hermitian manifolds, referring the reader to \cite{PavingTheWay, NiWallach, Wang} for further details.

Let $(M, I)$ be a compact complex manifold, and let $G$ be a compact connected Lie group acting effectively, transitively, and holomorphically on $M$. We fix a point $p\in M$, and set $L=G_p$, so that $M$ can be written as the coset space $G/L$. 
We also assume that $L$ is connected.

By compactness of $L$, we can decompose $\frakg$ as $\frakg = \frakl \oplus\frakm$, for an $\Ad(L)$-invariant subspace $\frakm$. The latter is identified with $T_pM$ by the usual isomorphism $\frakm\ni X\mapsto X^*_p$, where $X^*_p = \frac{\dd{}}{\dd{t}}\big|_0 \exp(tX)\cdot p$.
In particular, the complex structure $I$ on $M$ induces a complex structure, denoted by the same symbol, on $\frakm$.

There exists a subalgebra $\mathfrak{c} \subset \frakg$ which is the centralizer of an abelian subalgebra and verifies $\frakl\subset\mathfrak{c}$ and $[\mathfrak{c}, \mathfrak{c}] = [\frakl,\frakl]$.
It is possible to write $\mathfrak{c} = \frakl\oplus\frakt$, where $\frakt$ is abelian, and
\begin{equation*}
	\frakg = \frakl \oplus \frakt \oplus \frakn,
\end{equation*}
for $\frakn$ the orthogonal complement with respect to a $\Ad(G)$-invariant inner product.
Denote by $C$ the connected subgroup corresponding to $\mathfrak{c}$, which is the centralizer of a torus in $G$. Now, $F=G/C$ is a generalized flag manifold, endowed with the complex structure $\mathcal{I} = I|_{\frakn}$. If $L$ is connected, then $L\subset C$ and there is a holomorphic fibration $M \to F$, often called Tits fibration.

We can write $\frakg = \frakz(\frakg) \oplus \frakg_{ss}$, where $\frakz(\frakg)$ and $\frakg_{ss}$ are respectively the center and the semisimple part of $\frakg$. 
By construction, $\mathfrak{c}$ is contained in the normalizer of $\frakl$ in $\frakg$, so $[\frakl, \frakt]=0$, and one can select a maximal abelian subalgebra of $\frakg$ of the form $\frakt_{\frakl}\oplus\frakt$, where $\frakt_{\frakl}$ is a maximal abelian subalgebra of $\frakl$. 
Then $\frakt_{\frakl}\oplus\frakt = \frakz(\frakg) \oplus \fraks$, where $\fraks = (\frakt_{\frakl} \oplus \frakt) \cap \frakg_{ss}$ is clearly a maximal abelian subalgebra of $\frakg_{ss}$, so that $\frakh = \fraks^{\CC}$ is a Cartan subalgebra of $\frakg_{ss}^{\CC}$. 

In the following, $R$ will denote the root system of $\frakg_{ss}^{\CC}$ with respect to $\frakh$, and $\calB$ the Killing form. For every $\lambda \in \frakh^*$, the element $t_{\lambda} \in \frakh$ is defined by $\calB(t_{\lambda}, H) = \lambda(H)$ for every $H\in\frakh$. We set $(\lambda, \mu) = \calB(t_{\lambda}, t_{\mu})$. For each $\gamma\in R$, we fix root vectors $E_{\gamma}$ in the root space $\frakg_{\gamma}$ such that $\calB(E_{\gamma}, E_{-\gamma}) = 1$. This implies
\[
	H_{\gamma} := [E_{\gamma}, E_{-\gamma}] = t_{\gamma}.
\]
Moreover, without loss of generality we can assume $\bar{E}_{\gamma} = - E_{-\gamma}$, where the bar symbol denotes conjugation with respect to $\frakg$ (see \cite[pag. 619]{PavingTheWay}). Given $\gamma$, $\sigma \in R$ such that $\gamma+\sigma\in R$, the real numbers $N_{\gamma,\sigma}$ are defined by $[E_{\gamma}, E_{\sigma}] = N_{\gamma,\sigma}E_{\gamma+\sigma}$, and they satisfy the following useful relations:
\begin{equation}\label{E:N}
\begin{gathered}
	N_{\alpha,\beta} = N_{\beta,-\alpha-\beta} = N_{-\alpha-\beta, \alpha} =
		 - N_{\beta, \alpha} = -N_{-\alpha,-\beta},\\
	N_{\alpha + \beta, \gamma}N_{\alpha,\beta} = N_{\alpha+\gamma, \beta}N_{\alpha,\gamma} + N_{\beta,\gamma}N_{\alpha, \beta+\gamma},\\
	N_{\alpha,-\beta}^2 + N_{\alpha,\beta}^2 = (\alpha, \beta).
\end{gathered}
\end{equation}
The complexified Lie algebras can be decomposed as
\[
	\frakg^{\CC} 
	= \frakz(\frakg)^{\CC}
		\oplus 	\Big( 
						\frakh \oplus \bigoplus_{\gamma\in R} \frakg_{\gamma} 
				\Big),
	\quad
	\mathfrak{c}^{\CC} = 
				\Big( 
					\frakt_{\frakl}^{\CC} \oplus \bigoplus_{\gamma\in R_{\frakl}} \frakg_{\gamma} 
				\Big)
				\oplus \frakt^{\CC},
	\quad
	\mathfrak{l}^{\CC} = 
				\Big( 
					\frakt_{\frakl}^{\CC} \oplus \bigoplus_{\gamma\in R_{\frakl}} \frakg_{\gamma}
				\Big),
\]
where $R_{\frakl}$ is a subsystem of $R$. Letting $\hat{R} := R \setminus R_{\frakl}$, one obtains $\frakn^{\CC} = \bigoplus_{\gamma\in\hat{R}} \frakg_{\gamma}$.

Since $G/C$ is a generalized flag manifold and  $I|_{\frakn}$ coincides with $\mathcal{I}$, the subset $\hat{R}$ is symmetric (\cite{PavingTheWay}), in other words there is  an ``ordering'' $\hat{R} = \hat{R}^{+} \cup \hat{R}^{-}$  determined by
\[
	\frakn^{1,0} = \bigoplus_{\gamma\in\hat{R}^+} \frakg_{\gamma}, 
	\quad\quad
	\frakn^{0,1} = \bigoplus_{\gamma\in\hat{R}^-} \frakg_{\gamma}.
\]

Let $g$ be a $G$-invariant Riemannian metric on $M$. We can think of $g$ as an $\Ad(L)$-invariant scalar product on $\frakm$, compatible with the linear complex structure on $\frakm$. We will extend $g$ to $\frakm^{\CC}$ as a $\CC$-bilinear form. Recall that $g$ is said to be \emph{naturally reductive} with respect to the decomposition $\frakg = \frakl \oplus \frakm$ if we have
\begin{equation}\label{E:naturally}
	g([X,Y]_{\frakm} , Z) + g(Y, [X,Z]_{\frakm}) = 0
\end{equation}
for every $X,Y,Z \in \frakm$. In this case, it is easy to see that $g(E_{\gamma}, \bar{E}_{\sigma}) = 0$, for every $\gamma, \sigma \in \hat{R}^+$ with $\gamma\ne\sigma$.
If $L$ is trivial, so that $M = G$, a left invariant metric is naturally reductive if and only if it is bi-invariant.

By theory of homogeneous spaces (see \cite[Chapter X]{KobayashiNomizu}), an invariant connection $\nabla$ on $M=G/L$ can be identified with the linear map $\Lambda \colon \frakm \to \End(\frakm)$ given by
\[
	(\Lambda(X)Y)^*_p = - \left( A_{X^*} Y^* \right)_p
\]
for every $X,Y\in\frakm$. Here $A_{X^*}$ is the Nomizu operator associated to $\nabla$ and the fundamental vector field $X^*$, defined as
\[
	A_{X^*} = \mathcal{L}_{X^*}  - \nabla_{\!X^*}.
\]
The torsion $T^{\nabla}$ and the curvature $R^{\nabla}$ at the point $p$ are given by
\begin{equation}\label{E:tensors}
\begin{gathered}
	T^{\nabla}(X, Y) = \Lambda(X)Y - \Lambda(Y)X - [X, Y]_{\frakm},\\
	R^{\nabla}(X,Y) =
	[\Lambda(X), \Lambda(Y)]
	- \Lambda([X,Y]_{\frakm}) - \ad([X,Y]_{\frakl}),
\end{gathered}
\end{equation}
for all $X,Y\in\frakm$.
Conversely, any linear map $\Lambda \colon \frakm \to \End(\frakm)$ satisfying appropriate conditions yields an invariant connection on $M$.
The connection corresponding to $\Lambda \equiv 0$ is called \emph{canonical connection}, and it  preserves any $G$-invariant tensor. Its torsion is $T(X,Y) = - [X,Y]_{\frakm}$ for every $X,Y \in \frakm$. In particular, given an invariant Hermitian structure $(g, I)$ on $M$ as above, with $g$ naturally reductive, the canonical connection is Hermitian. Its torsion $3$-tensor is given by
\[
	c(X,Y,Z) = - g([X,Y]_{\frakm},Z), \quad \forall\, X,Y,Z\in \frakm,
\]
hence it is totally skew-symmetric. It follows that the canonical connection coincides with the Bismut connection of $(g,I)$. 
In the case of Lie groups, it is easy to check that $c$ is also closed, which means that $g$ is SKT.

\subsection{Joyce hypercomplex structures}\label{SS:Joyce}
We briefly recall the construction, due to Joyce in \cite{Joyce}, of compact homogeneous spaces equipped with invariant hypercomplex structures. 
The idea was already present in the physics literature, see~\cite{SSTV}.

Consider a connected, compact Lie group $G$, with a fixed maximal torus. 
Let $R$ be the corresponding root system for the semisimple part of $\frakg^{\CC}$, and choose an ordering $R = R^+ \cup R^-$.
For each root $\alpha$, we denote by $\fraks_{\alpha}\simeq\mathfrak{sl}(2,\CC)$ the subalgebra of $\frakg^{\CC}$ generated by $\frakg_{\alpha}$ and $\frakg_{-\alpha}$. 
Let $\alpha_1$ be a maximal root in $R^+$. Let us define
\[
	R^{+}_1 = \{ \gamma\in R^+ \colon (\alpha_1, \gamma) \ne 0 \}, \quad
	\Theta_1 = \{ \gamma\in R \colon (\alpha_1, \gamma) = 0 \},
\]
and inductively for $j \ge 2$,
\[
	R^{+}_j = \{ \gamma\in \Theta_{j-1}\cap R^+ \colon (\alpha_j, \gamma) > 0 \}, \quad
	\Theta_j = \{ \gamma\in \Theta_{j-1} \colon (\alpha_j, \gamma) = 0 \},
\]
where each root $\alpha_j$, for $j \ge 2$, is chosen to be maximal in the subsystem $\Theta_{j-1}$. 
There is a recursive decomposition of $\frakg$ which at iteration $m$ has the form
\[	
	\frakg 	= \frakb_m 
			\oplus  	\bigoplus_{j=1}^m \frakd_j 
			\oplus  	\bigoplus_{j=1}^m \frakf_j ,
\]
where
\[
	\frakd_j = 	\frakg \cap \fraks_{\alpha_j},
	\qquad
	\frakf_j =		\frakg \cap \bigoplus_{
								\substack{
											\beta \in R_j^+ \\ 
											\beta\ne\alpha_j
								}			
				} \frakg_{\beta} \oplus \frakg_{-\beta} 	,
	\qquad
	\frakb_m = 	\frakg\cap\left(
				\bigcap_{j=1}^m \mathfrak{Z}_{\frakg^{\CC}} (\fraks_{\alpha_j})  
				\right) .
\]
We will refer to this as a \emph{Joyce decomposition} of $\frakg$. 
The iteration can be performed until a step $d$ is reached in which $\frakb_d$ is abelian.
Note that $\frakz(\frakg) \subset \frakb_d$.
The roots $\alpha_j$ are \emph{strongly orthogonal roots}, namely $\alpha_i\pm\alpha_j \notin R\cup\{0\}$ for every $i\neq j$, and $\{\alpha_1,\dots,\alpha_d\}$ is a maximal strongly orthogonal subset of $R$.
Since $\alpha_j$ is maximal in $\Theta_{j-1}$, it is well know that $-1 \leq 2\frac{(\gamma, \alpha_j)}{(\alpha_j, \alpha_j)} \leq 1$ for every $\gamma\in\Theta_{j-1}$, $\gamma\neq\alpha_{j}$, so $\gamma\in R_j^+$ if and only if $2\frac{(\gamma, \alpha_j)}{(\alpha_j, \alpha_j)} = 1$.

Being $G$ compact, on $\frakg$ we can fix an $\Ad(G)$-invariant, negative definite scalar product, which coincides with $\calB$ on $\frakg_{ss}$. The Joyce decomposition is clearly orthogonal with respect to it. It should be noted that two different decompositions are related by an inner automorphism of $\frakg$.

Now suppose $L\subset G$ is a closed subgroup, and let $M=G/L$. Assume there exists a maximal torus of $G$ and a set of strongly orthogonal roots such that, for the corresponding Joyce decomposition, and for an integer $m\leq d$, the Lie algebra of $L$ satisfies the following properties:
\begin{equation}\label{E:decl}
	\frakl = 	(\frakb_d \cap \frakl) \,
			\oplus  	\bigoplus_{j=m+1}^d \!\frakd_j 
			\oplus  	\bigoplus_{j=m+1}^d \!\frakf_j ,
	\quad\quad
	\dim\frakb_m - \dim\frakl \equiv m \mod 4.
\end{equation}
Note that $\frakl \subset \frakb_m$, the semisimple part of $\frakl$ coincides with that of $\frakb_m$, and the center of $\frakl$ is contained in the center of $\frakb_m$.
Denoting by $\frakm$ the orthogonal complement of $\frakl$, it is proven in \cite{DimitrovTsanovII} that
\begin{equation}\label{E:decm}
	\frakm = (\frakb_d \cap \frakm)
			\oplus  	\bigoplus_{j=1}^m \frakd_j 
			\oplus  	\bigoplus_{j=1}^m \frakf_j .
\end{equation}
Setting $\mathfrak{v}:=\frakb_d \cap \frakl$, $\mathfrak{u}:=\frakb_d \cap \frakm$, it is also proved  that $\frakb_d = \mathfrak{v}\oplus\mathfrak{u}$ .

A hypercomplex structure on $\frakm$ is constructed in the following way.
For every $j = 1,\dots, m$, we pick an isomorphism $\frakd_j \simeq \frsu(2)$, which is clearly determined by the choice of  a basis $\{ X_2^j, X_3^j, X_4^j \}$ of $\frakd_j$ satisfying
\begin{equation}\label{E:bracket}
	[X_2^j, X_3^j] = 2X_4^j, \quad
	[X_3^j, X_4^j] = 2X_2^j, \quad
	[X_4^j, X_2^j] = 2X_3^j.	
\end{equation}
We select $m$ linearly independent vectors $X_1^1, \dots, X_1^m \in \fraku$, so that $\fraku = \fraku_0 \oplus \langle X_1^1, \dots, X_1^m \rangle$ for a subspace $\fraku_0 \simeq\RR^{4k}$.
We define
\[
	\frakm_j := \RR X_1^j \oplus \frakd_j \oplus \frakf_j \quad \forall \, j=1, \dots, m.
\]
It is easy to show that $\ad(\frakl)(\frakm_j) \subset \frakd_j\oplus\frakf_j$ for every $j = 1,\dots,m$.
In particular, if $L$ is connected, each subspace $\frakm_j$ is an $\Ad(L)$-submodule of $\frakm$.
Note that $\frakm = \fraku_0 \oplus\bigoplus_{j=1}^{m} \frakm_j$.
Now $I$ and $J$ are  defined on $\frakm$ as follows
\begin{enumerate}
\item[(i)]\label{E:Joyce1}
For every $j = 1,\dots,m$, let
\[
	IX_1^j = X_2^j,	\quad
	IX_3^j = X_4^j,	\quad
	JX_1^j = X_3^j,	\quad
	JX_2^j = - X_4^j,
\]
and $I^2 = J^2 
= - \id$.
\item[(ii)]\label{E:Joyce2}
For every $j=1,\dots,m$, let $I,J$ act on $\frakf_j$ as
\[
	I|_{\frakf_j} = \ad(X_2^j), \quad
	J|_{\frakf_j} = \ad(X_3^j). \quad
\]
\item[(iii)]
On $\fraku_0$, the complex structures $I$ and $J$ are given by the choice of an isomorphism $\fraku_0 \simeq \mathbb{H}^k$.
\end{enumerate}
It is shown by Joyce in~\cite{Joyce} that this is a well defined $\ad(\frakl)$-invariant hypercomplex structure on $\frakm$ and, if $L$ is connected, that it gives rise to an invariant (integrable) hypercomplex structure on $M=G/L$, using results by Wang (\cite{Wang}).

On the other hand, one can consider a generic homogeneous hypercomplex manifold $(M=G/L, I, J)$ where $G$ is a connected compact Lie group and $L$ a closed subgroup (not necessarily connected), with the action of $G$ on $M$ being effective.
It is proved in \cite{DimitrovTsanovII} that $I$ and $J$ arise from the above construction (see also \cite{BedulliGoriPodesta}).
More precisely, considering first the homogeneous complex manifold $(G/L, I)$, one can choose a Cartan subalgebra $\frakh\subset\frakg_{ss}^{\CC}$ and a splitting of the corresponding root system $R=R_{\frakl} \cup \hat{R}$, $\hat{R} = \hat{R}^+ \cup \hat{R}^-$, as in section \ref{SS:homogeneous}. Then there is a set of strongly orthogonal roots $\{ \alpha_1, \dots ,\alpha_d \}$ such that, considering the Joyce decomposition corresponding to these data, one has $\frakl\subset\frakb_m$ and $\dim\frakb_m - \dim\frakl \equiv m \mod 4$.
In the notations of Section~\ref{SS:homogeneous}, the following relations hold:
\begin{equation}\label{E:mainEquation}
\begin{aligned}
	IE_{\beta} 	&= iE_{\beta}, \quad \forall\,\beta\in R_j^+,\\
	JE_{\alpha_j} 	&= k_j(H_{\alpha_j} + iIH_{\alpha_j}),\\
	JE_{\gamma} 	&= 2k_j N_{\gamma,-\alpha_j} E_{\gamma - \alpha_j},
		\quad \forall\,\gamma\in R_j^+, \, \gamma\ne\alpha_j,
\end{aligned}
\end{equation}
where $k_1, \dots, k_m$ are complex numbers satisfying $\abs{k_j}^2 = \frac{1}{2\abs{\alpha_j}^2}$.
In particular, setting $R_j^- = - R_j^+$ and $R^j = R_j^+ \cup R_j^-$ we have $\hat{R}^+ = \bigcup_j R_j^+$ and $\hat{R}^- = \bigcup_j R_j^-$. Defining
\begin{equation}\label{E:mainbasis}
\begin{aligned}
	X_1^j& 	= - \frac{2i}{\abs{\alpha_j}^2} IH_{\alpha_j},&					
	X_2^j&	= 	\frac{2i}{\abs{\alpha_j}^2} H_{\alpha_j},& 		\\
	X_3^j&	= 	2(	\overline{k}_j E_{\alpha_j} - k_j E_{-\alpha_j}	),&		\qquad
	X_4^j&	= 	2i(	\overline{k}_j E_{\alpha_j} + k_j E_{-\alpha_j}	),&
\end{aligned}
\end{equation}
then 
$I$ and $J$ can be given via the same relations given by Joyce.
These vectors yield a basis of each $\frakd_j$ satisfying \eqref{E:bracket}.\\

From now on, we shall always assume that $\dim\frakb_m - \dim\frakl = 4$ in decomposition \eqref{E:decl}. 
With this assumption, we call $(G/L,I,J)$ a \emph{Joyce hypercomplex manifold}.
In particular we have $\frakm = \bigoplus_{j=1}^{m} \frakm_j$.
Moreover we shall choose a Cartan subalgebra  and a set of strongly orthogonal roots as described above, in order to have relations~\eqref{E:mainEquation} on the corresponding Joyce decomposition.

We stress that different isomorphisms $\frakd_j \simeq\frsu(2)$ lead to equivalent hypercomplex structures, but the choice of the identification $\fraku\simeq\RR^{m}$ is essential (see~\cite{GentiliEtAl}).

\begin{remark}\label{R:simple}
Note that the Joyce decomposition is compatible with the decomposition $\frakg = \frakz(\frakg) \oplus \frakg_1\oplus\dots\oplus\frakg_t$, where $\frakg_k$ are simple ideals and $\frakz(\frakg)$ is the center. 
Considering the splitting of the root system $R$ into irreducible components, $R = \Delta_1 \sqcup\dots\sqcup \Delta_t$, it is easy to check that $\alpha_j \in \Delta_k$ for some $k$ implies $R_j \subset \Delta_k$. In particular $\frakd_j \oplus \frakf_j \subset \frakg_k$. 
\end{remark}

\section{Invariant HKT metrics}\label{S:HKT}
In this section, we consider a Joyce hypercomplex manifold $(M= G/L, I, J)$, and we study invariant HKT metrics on it, establishing Theorem~\ref{T:HKT}.
Initially we do not require $L$ to be connected.

As explained in section~\ref{SS:Joyce}, we will assume that  formulas~\eqref{E:mainEquation} hold on each layer $\frakm_j$. 
Let $g$  be a $G$-invariant hyperhermitian metric on $M$.
We note that the HKT condition $\partial \Omega = 0$ can be rewritten as
\[
	\dd{\Omega (X,Y,Z)} = 0 \quad  \text{ for every $X,Y,Z$ of type $(1,0)$ w.r.t. $I$.}
\]
From now on, we shall identify $g$ and $\Omega$ with the corresponding tensors on $\frakm^{\CC}$. Since $\Omega(U, V) = - g(U, JV)$ for every $U,V\in\frakm^{1,0}$, one has
\begin{equation}\label{E:diffOmega}
\begin{aligned}
	\dd{\Omega(X,Y,Z)} 
	&=	- \Omega( [X,Y]_{\frakm^{10}}, Z )
		- \Omega( [Z,X]_{\frakm^{10}}, Y )
		- \Omega( [Y,Z]_{\frakm^{10}}, X )\\
	&=	g( [X,Y]_{\frakm^{10}}, JZ ) +
		g( [Z,X]_{\frakm^{10}}, JY ) +
		g( [Y,Z]_{\frakm^{10}}, JX ),
\end{aligned}
\end{equation}
so $g$ is HKT if and only if
\begin{equation}\label{E:dOmega}
	g( [X,Y]_{\frakm^{10}}, JZ ) +
	g( [Z,X]_{\frakm^{10}}, JY ) +
	g( [Y,Z]_{\frakm^{10}}, JX ) = 0
	\quad \forall\, X,Y,Z \in \frakm^{1,0}.
\end{equation}

Let us assume that $g$ is actually HKT. We use formulas~\eqref{E:mainEquation} to deduce the following necessary conditions.
For brevity, let us denote $H_j = H_{\alpha_j} - i I H_{\alpha_j}$, for every $j\in\{ 1, \dots, m \}$.

\begin{lemma}\label{L:orthogonality}
Let $j\in\{ 1, \dots, m \}$. Then:
\begin{enumerate}
\item[(i)]
$g( E_{\beta}, \bar{H}_j ) = 0 		\quad \forall\, \beta \in R_j^+$;
\item[(ii)]
$g( E_{\beta}, E_{-\gamma} ) = 0	\quad \forall\, \beta,\gamma \in R_j^+, \, \beta\ne\gamma$;
\item[(iii)]
$g( E_{\beta}, E_{-\beta} ) = g( E_{\alpha_j}, E_{-\alpha_j} ) 	\quad \forall\, \beta \in R_j^+$.
\end{enumerate}
\end{lemma}
\begin{proof}
For any $\beta\in R_j^+$, $\beta \ne \alpha_j$, consider relation~\eqref{E:dOmega} with $(X,Y,Z) = (H_j, E_{\beta}, E_{\alpha_j})$. Since $\alpha_j + \beta$ is not a root, we get
\[
	0 =
	\beta(H_j)g(E_{\beta}, JE_{\alpha_j}) - \alpha_j(H_j)g(E_{\alpha_j}, JE_{\beta}) =
	(\beta+\alpha_j)(H_j) g(E_{\beta}, \bar{H}_j).
\]
If $g(E_{\beta}, \bar{H}_j) \ne 0$, then $0 = (\beta + \alpha_j)(H_j) = (\beta + \alpha_j, \alpha_j) - (\beta + \alpha_j)( iIH_{\alpha_j})$; 
since $(\beta + \alpha_j, \alpha_j)$ is the real part of this complex number, we would have
$0 = (\beta + \alpha_j, \alpha_j) = (\beta, \alpha_j) + (\alpha_j, \alpha_j) > 0$, a contradiction. 
Note that also $g(E_{\alpha_j}, \bar{H}_j) = 0$, since $\bar{H}_j$ is proportional to $JE_{\alpha}$, thus (i) is proved.

To show part (ii), let $\beta,\gamma \in R_j^+$ be such that $\beta \ne \gamma$. 
We can assume both $\beta$ and $\gamma$ differ from $\alpha_j$, otherwise applying $J$ we are in the previous case.
Set $\sigma = \alpha_j - \gamma \in R_j^+$.
We note $\beta + \sigma$ is not a root, otherwise we would have $\beta + \sigma = \alpha_j$ (by \cite[Prop. 2.2]{DimitrovTsanovI}) hence $\beta = \gamma$.
Therefore relation~\eqref{E:dOmega} with
$(X,Y,Z) = ( H_j, E_{\beta}, E_{\sigma} )$ gives
\[
	(\beta + \sigma)(H_j) g(E_{\beta}, E_{-\gamma}) = 0.
\]
If $g(E_{\beta}, E_{-\gamma}) \ne 0$, then $0 = (\beta + \sigma)(H_j) = (\alpha_j, \alpha_j) + (\beta - \gamma, \alpha_j) - (\beta + \sigma)( iIH_j )$, and taking the real part, $(\alpha_j, \alpha_j) + (\beta - \gamma, \alpha_j) = 0$.
But by conjugation we also have  $g(E_{\gamma}, E_{-\beta}) \ne 0$, hence $(\alpha_j, \alpha_j) + (\gamma - \beta, \alpha_j) = 0$.
Summing the two items yields $(\alpha_j, \alpha_j) = 0$, a contradiction. This proves claim (ii).

For part (iii), given $\beta\in R_j^+$, $\beta\ne\alpha_j$, we set $\rho = \alpha_j - \beta \in R_j^+$, and consider once again \eqref{E:dOmega} with $(X,Y,Z) = (H_j, E_{\beta}, E_{\rho})$:
\[
	\alpha_j(H_j)g(E_{\beta}, JE_{\rho}) - \frac{N_{\beta,\rho}}{\overline{k}_j}g(E_{\alpha_j}, E_{-\alpha_j}) = 0.
\]
Since
$JE_{\rho} = 2k_j N_{\rho, -\alpha_j} E_{\rho - \alpha_j} = 2k_j N_{\beta, \rho}E_{-\beta}$,
\[
\begin{aligned}
0	&= 	\alpha_j(H_j) 2k_j N_{\beta, \rho} g(E_{\beta}, E_{-\beta}) 
			- \frac{N_{\beta, \rho}}{\overline{k}_j} g(E_{\alpha_j}, E_{-\alpha_j}) \\
	&=	\frac{N_{\beta, \rho}}{\overline{k}_j}
			\left(
			(\alpha_j,\alpha_j)2\abs{k_j}^2  g(E_{\beta}, E_{-\beta})
			- g(E_{\alpha_j}, E_{-\alpha_j}) 
			\right)  \\
	&=	\frac{N_{\beta, \rho}}{\overline{k}_j} 	
			\left(
			g(E_{\beta}, E_{-\beta}) - g(E_{\alpha_j}, E_{-\alpha_j}) 
			\right),
\end{aligned}
\]
where we used that $\abs{k_j}^2 = \frac{1}{2\abs{\alpha_j}^2}$. It follows that $g(E_{\beta}, E_{-\beta}) = g(E_{\alpha_j}, E_{-\alpha_j})$, and the proof is complete.
\end{proof}

\begin{lemma}\label{L:levels}
Let $j,l \in \{1,\dots,m\}$ with $j\ne l$. Then $g( \frakm_j, \frakm_l ) = 0$.
\end{lemma}
\begin{proof}
We first use \eqref{E:dOmega} for $(X,Y,Z) = (H_j, H_l, E_{\alpha_j})$, obtaining
\[
	- \frac{1}{\overline{k}_l} \alpha_j(H_j) g(E_{\alpha_j}, E_{-\alpha_l})
	+ \frac{1}{\overline{k}_j} \alpha_j(H_l) g(E_{\alpha_j}, E_{-\alpha_j}) = 0.
\]
Now $\alpha_j(H_j) = (\alpha_j, \alpha_j) > 0$, while $\alpha_j(H_l) = 0$, hence we are left with $g(E_{\alpha_j}, E_{-\alpha_l}) = 0$. It follows from this that $g(H_j, \bar{H}_l) = 0$.

Let $\beta \in R_j^+$. Relation~\eqref{E:dOmega} with $(X,Y,Z) = (H_j, E_{\beta}, E_{\alpha_l})$ yields $(\beta + \alpha_l)(H_j) g(E_{\beta}, \bar{H}_l) = 0$. Since $(\beta, \alpha_j) > 0$, $(\beta + \alpha_l)(H_j)$ is not vanishing, hence $g(E_{\beta}, \bar{H}_l) = 0$. 
It is easy to obtain from this also that $g(E_{\beta}, E_{-\alpha_l}) = 0$.

Finally, let $\beta \in R_j^+$ and $\sigma\in R_l^+$ with $\sigma\ne\alpha_l$. Set $\gamma = \alpha_l - \sigma \in R_l^+$. Without loss of generality we can assume $l > j$.
If $\beta + \gamma$ is a root, it belongs to $R_j^+$, so that $g(E_{\beta+\gamma}, E_{-\alpha_l}) = 0 $ by the previous step. Therefore, relation~\eqref{E:dOmega} with $(X,Y,Z) = (H_j, E_{\beta}, E_{\gamma})$ reduces to
\[
	(\alpha_l + \beta - \sigma)(H_j) g(E_{\beta}, E_{-\sigma}) = 0.
\]
If $g(E_{\beta}, E_{-\sigma}) \ne 0$, then $(\beta - \gamma)(H_j) = (\alpha_l + \beta - \sigma)(H_j) =  0$. Considering the real part, we have $(\beta, \alpha_j) - (\gamma, \alpha_j) = 0$. But $(\gamma, \alpha_j) = 0$, since we are assuming $l > j$, thus $(\beta, \alpha_j) = 0$, which is false. We conclude that $g(E_{\beta}, E_{-\sigma}) = 0$.
\end{proof}

We note that, by Lemma~\ref{L:orthogonality}, on the subspace $\frakd_j \oplus \frakf_j$, $g$ coincides with the restriction of the negative of the Killing form, scaled by the factor $g_j = g(E_{\alpha_j}, \bar{E}_{\alpha_j}) > 0$.
At this point, we can introduce the following scalar product $h$ on each $\ad(\frakl)$-module $\frakm_j$:
\begin{equation}\label{E:h}
	h|_{\frakd_j \oplus \frakf_j} = - \calB|_{\frakd_j \oplus \frakf_j}, \quad
	h( X_1^j, \frakd_j \oplus \frakf_j) = 0, \quad
	h(X_1^j, X_1^j) = \frac{4}{\abs{\alpha_j}^2}.
\end{equation}
Then $h$ is extended to $\frakm$ setting the different $\frakm_j$'s to be $h$-orthogonal.


\begin{proof}[Proof of Theorem \ref{T:HKT}]
The first part follows from Lemmas~\ref{L:orthogonality} and~\ref{L:levels}, where $h$ is the scalar product on $\frakm$ defined by~\eqref{E:h}.

For the proof of the second part, we assume the stabilizer $L$ to be connected and, conversely, we show that any scalar product $g$ defined by~\eqref{E:scalarproduct} induces an invariant HKT metric on $M$.
Since $h$ coincides with $-\calB$ on $\frakd_j\oplus\frakf_j$, one can check that $g$ is $\ad(\frakl)$-invariant.
The proof that $g$ is compatible with $I$ and $J$ is straightforward considering the complexification and using relations~\eqref{E:mainEquation}.
In order to show that $g$ is HKT, there are three steps to be addressed.

First, let $\alpha\in R_j^+$, $\beta\in R_k^+$, $\gamma\in R_l^+$, with $l\ge k\ge j$.
We assume $\alpha+\beta$, $\alpha+\gamma$ and $\beta+\gamma$ are roots, the other cases being similar.
By \eqref{E:diffOmega},
\[
	\dd{\Omega}(E_{\alpha}, E_{\beta}, E_{\gamma}) =
	N_{\alpha,\beta} g(E_{\alpha+\beta }, JE_{\gamma}) + 
	N_{\gamma,\alpha} g(E_{\alpha+\gamma}, JE_{\beta}) + 
	N_{\beta,\gamma} g(E_{\beta+\gamma} , JE_{\alpha}).
\]
This can be possibly non-zero only when $l=k=j$ and $\alpha+\beta+\gamma = \alpha_j$, in which case we have
\[
\begin{split}
	\dd{\Omega}(E_{\alpha}, E_{\beta}, E_{\gamma}) 
	&	=			2k_j N_{\gamma, -\alpha_j} N_{\alpha,\beta} g(E_{\alpha+\beta} , E_{-\alpha-\beta})\\
	&	\quad 	+ 	2k_j N_{\beta, -\alpha_j} N_{\gamma,\alpha} g(E_{\alpha+\gamma}, E_{-\alpha-\gamma})\\
	&	\quad 	+ 	2k_j N_{\alpha, -\alpha_j} N_{\beta,\gamma} g(E_{\beta+\gamma} , E_{-\beta-\gamma})\\
	&	=	- 2k_jg_j \left(
		N_{\alpha+\beta, \gamma}N_{\alpha,\beta} 
		- N_{\alpha+\gamma, \beta}N_{\alpha,\gamma} 
		- N_{\alpha, \beta+\gamma}N_{\beta,\gamma}
		\right) = 0,
\end{split}
\]
where we used~\eqref{E:N}.
Second, let $\beta\in R_k^+$, $\gamma\in R_l^+$, with $l\ge k$.
If $\beta + \gamma$ is not a root, $\dd{\Omega}(H_j, E_{\beta}, E_{\gamma}) = (\beta + \gamma)(H_j) g(E_{\beta}, JE_{\gamma}) = 0$.
Otherwise, 
\[
	\dd{\Omega}(H_j, E_{\beta}, E_{\gamma}) 
	= (\beta + \gamma)(H_j) g(E_{\beta}, JE_{\gamma}) 
	- \frac{N_{\beta,\gamma}}{\overline{k}_j} g(E_{\beta+\gamma}, E_{-\alpha_j}).
\]
The only non-trivial case to check is for $l=k$ and $\beta+\gamma = \alpha_j$.
Assuming this, just like in the proof of Lemma~\ref{L:orthogonality}, one can compute
\[
\begin{aligned}
	\dd{\Omega}(H_j, E_{\beta}, E_{\gamma})
	&=	\alpha_j(H_j) 2k_j N_{\beta, \gamma} g(E_{\beta}, E_{-\beta}) 
		- \frac{N_{\beta, \gamma}}{\overline{k}_j} g(E_{\alpha_j}, E_{-\alpha_j})\\
	&=	\frac{N_{\beta, \gamma}}{\overline{k}_j} (-g_j + g_j) = 0.
\end{aligned}
\]
Finally, consider the triple $(H_j, H_l, E_{\gamma})$, with $l>j$ and $\gamma\in R_k^+$.
Using~\eqref{E:diffOmega} we find
\[
\begin{aligned}
	\dd{\Omega}(H_j, H_l, E_{\gamma}) 
	&= - \gamma(H_j) g(E_{\gamma}, JH_l) + \gamma(H_l) g(E_{\gamma}, JH_j)\\
	&= \frac{1}{\overline{k}_l} \gamma(H_j) g(E_{\gamma}, E_{-\alpha_l}) - \frac{1}{\overline{k}_j} \gamma(H_l) g(E_{\gamma}, E_{-\alpha_j}).
\end{aligned}
\]
If $\gamma$ is different from both $\alpha_j$ and $\alpha_l$, the expression vanishes.
Otherwise, for instance when $\gamma = \alpha_l$ (and $\gamma \ne \alpha_j$), we have
\[
	\dd{\Omega}(H_j, H_l, E_{\gamma}) 
	= - \frac{1}{\overline{k}_l} \alpha_l(H_j) g_l = 0,
\]
since $\alpha_l(H_{\alpha_j}) = (\alpha_j, \alpha_l) = 0$ and $\alpha_l(IH_{\alpha_j}) = 0$.\\
\end{proof}
It is known that a naturally reductive metric $g$ on $G/L$, which is also hyperhermitian with respect to $I,J$, is HKT: the Bismut connections of the Hermitian structures $(g,I)$, $(g,J)$ both coincide with the canonical connection of the homogeneous space.
This was initially observed in~\cite{OpfermannPapadopoulos} and~\cite{GrantcharovPoon}, where in particular a naturally reductive metric is constructed starting from the negative of the Killing form of $\frakg_{ss}$.  However, it has been pointed out in~\cite[Section 2.3]{GentiliEtAl} that a given hypercomplex structure may be not compatible with extensions of the Killing form.


If in the Joyce decomposition one has $\frakb_d = \frakz(\frakg)$, by construction $X_1^1, \dots, X_1^m \in \frakz(\frakg)$. 
One can extend $h|_{\mathfrak{u} \times \mathfrak{u}}$ to a scalar product on $\frakz(\frakg) = \frakb_d = \frakv\oplus\fraku$, say $\tilde{h}_{\frakz}$, in such a way that $\mathfrak{u}$ and $\mathfrak{v}$ are orthogonal. 
Then $\tilde{h} = \tilde{h}_{\frakz} + (-\calB)|_{\frakg_{ss}}$ is an $\Ad(G)$-invariant scalar product on $\frakg$ for which $\frakl$ and $\frakm$ are orthogonal, and $h = \tilde{h}|_{\frakm\times\frakm}$. Consequently $h$ is naturally reductive in this case.

In case $L$ is trivial, then $h$ is a bi-invariant metric on $G$ provided $\frakb_d = \frakz(\frakg)$.
As noted in \cite[Sec. 5]{OpfermannPapadopoulos}, if all the simple factors of $\frakg$ are of type $\operatorname{B_{\ell}, C_{\ell}, D_{2\ell}, E_7, E_8, F_4, G_2}$, then $\frakb_d = \frakz(\frakg)$.

\begin{remark}\label{R:flag}
Let $(G/L, I, J, g)$ be a Joyce hypercomplex manifold, $L$ connected, with an invariant HKT metric $g$. 
Consider the Hermitian manifold $(G/L, I, g)$ and the fibration over the generalized flag manifold $(G/C, \mathcal{I})$ mentioned in section~\ref{SS:homogeneous}.
By the first part of Theorem~\ref{T:HKT}, the restriction $g|_{\frakn\times\frakn}$ is $\Ad(C)$-invariant, hence $g$ is projectable to a Hermitian metric $\tilde{g}$ on $(G/C, \mathcal{I})$.
We note that $\tilde{g}$ is never Kähler. Indeed, for any index $1\leq j \leq m$, take $\alpha \in R_j^{+}$, $\alpha \ne \alpha_j$, and let $\beta = \alpha_j - \alpha \in R_j^+$. 
Then $\alpha + \beta = \alpha_j$, and $\tilde{g}(E_{\alpha}, \bar{E}_{\alpha}) = \tilde{g}(E_{\beta}, \bar{E}_{\beta}) = \tilde{g}(E_{\alpha+\beta}, \bar{E}_{\alpha+\beta}) = g_j$, hence $\tilde{g}$ is not Kähler by the characterization given in~\cite{AP}.
\end{remark}

\begin{example}\label{Example}
Let us explore in more depth the case of homogeneous hypercomplex manifolds $G/L$ with $G=\SU(n)$, $n \ge 3$. 

We take the usual Cartan subalgebra of $\mathfrak{sl}(n, \CC)$ consisting of traceless diagonal matrices.
The root system is $R=\{ \vareps_j - \vareps_k \colon \text{$1 \leq j,k \leq n$,  $j \ne k$} \}$, where $\vareps_j$ extracts the $j$th diagonal entry of the matrix.
The number of layers in the ``full'' Joyce decomposition is $d = \frac{n-1}{2}$ if $n$ is odd, and $d = \frac{n}{2}$ if $n$ is even, 
and we select the strongly orthogonal roots as 
\[
	\alpha_1 = \vareps_1 - \vareps_2, 
	\quad \alpha_2 = \vareps_3 - \vareps_4, 
	\quad \dots \quad \alpha_d = \vareps_{2d-1} - \vareps_{2d}.
\]
We have
\[
	R_j^+ = \{ \alpha_j \}
			\cup \{ \vareps_{2j - 1} - \vareps_p \colon 2j+1 \leq p \leq n \}
			\cup \{ \vareps_q - \vareps_{2j} \colon 2j+1 \leq q \leq n \},
\]
for $j= 1,\dots,d$. 
The root vectors are $E_{\vareps_j - \vareps_k} = \frac{1}{\sqrt{2n}}e_{jk}$, where $e_{jk}$ are elementary matrices, and $H_{\vareps_j - \vareps_k} = \frac{1}{2n}(e_{jj} - e_{kk})$.
Now if $n$ is odd $\frakb_d$ is spanned by
\[
\begin{aligned}
	A_1 &= \diag	\left(
				i, i, - \frac{2i}{n-2}, \dots, - \frac{2i}{n-2}
				\right),\\
	&\cdots\\
	A_j &= \diag	\left(
				0, \dots, 0, i, i, - \frac{2i}{n-2j}, \dots, - \frac{2i}{n-2j}
				\right),\\
	&\cdots\\
	A_{d-1} &= \diag	
				\left(
				0,\dots, 0, i, i, -i, -i
				\right),\\
	A_d &= \diag	\left(
				0,\dots, 0, i, i, -2i
				\right),			
\end{aligned}
\]
hence it has dimension $d$.
If $n$ is even, then $\frakb_d$ is spanned by $A_1,\dots,A_{d-1}$, so its dimension is $d-1$.
In terms of matrices the corresponding Joyce decomposition of $\frakg = \frsu(n)$ when $n$ is odd is given  in~\cite[Sec. 4.3]{GentiliEtAl}. (For $n$ even it is similar but there is a copy of $\frsu(2)$ embedded in the right bottom corner.)

The subgroup $L$ can be chosen for any $n$ as
\[
	L = 	
		\left\{
			\begin{bmatrix}
			\id_{2m}		&	\\
						&\ell
			\end{bmatrix}
			\colon
			\ell \in \SU(n-2m) 
		\right\},
\]
where $m$ is such that $2 \leq 2m < n$.  
It is proved in~\cite[Sec.~5]{DimitrovTsanovII} that products of such hypercomplex manifolds exhaust all the simply connected Joyce hypercomplex manifolds.
If $L$ is not trivial, the subspace $\fraku$ appearing in~\eqref{E:decl} is spanned by $A_1,\dots, A_m$.

We can produce examples of homogenous hypercomplex manifolds with disconnected isotropy group. Indeed, let $G/L$ be one of the homogenous manifolds described above ($L$ possibly trivial),
and let
\[
	a =
	\begin{bmatrix}
	-\id_{2}	&			\\				
			&\id_{n-2}
	\end{bmatrix}
	\in \SU(n).
\]
Clearly $L' := L \cup aL$ is a closed subgroup of $G$ with two connected components.
On elementary matrices $e_{pq}$, one can check that for $p=1,2$ and $3 \leq q \leq n$
\[
	\Ad(a)e_{pq}= - e_{pq}, \qquad \Ad(a)e_{qp}= - e_{qp}.
\] 
The other elementary matrices are kept fixed.
Thus: $\Ad(a) X_1^j = X_1^j$ for every $j$, since $X_1^j$ are diagonal;
$\Ad(a)|_{\frakd_j} = \id$ for every $j$;
$\Ad(a)|_{\frakf_1} = - \id$ and $\Ad(a)|_{\frakf_j} = \id$ if $j\geq2$.
It follows that $I$ and $J$ are $\Ad(a)$-invariant, as well as every scalar product $g$ of the form~\eqref{E:scalarproduct}.
We have then invariant hypercomplex structures  and compatible HKT metrics on $G/L'$.
\end{example}

\section{Invariant HKT-Einstein metrics}\label{S:HKT-Einstein}
Consider a Joyce hypercomplex manifold $(G/L, I, J)$, with an invariant HKT metric $g$, which has form described in Theorem~\ref{T:HKT}.
We need to determine the components of the (first) Chern-Ricci form associated to $(I,g)$.
Even though this is essentially known, for completeness we outline the computation, using the linear map $\Lambda \colon \frakm \to \End(\frakm)$ associated to the Chern connection.
Note that the computation is valid for every homogeneous Hermitian manifold $(G/L, I, g)$, such that $g$ is projectable on the base of the fibration of $G/L$ over the flag manifold $G/C$.


We keep the same notations as in the previous sections.
For $X\in\frakm$, taking into account that  the vector field $X^* $ is Killing and that $\nabla^{\Ch}$ is compatible with $g$, we see that $A_{X^*}g  = 0$. It follows that each endomorphism $\Lambda(X)$ is skew-adjoint with respect to $g$.
We shall extend $\Lambda$ to a complex-linear map $\frakm^{\CC} \to \End(\frakm^{\CC})$. 
For $Z \in \frakm^{\CC}$, $\Lambda(Z)$ is skew-adjoint with respect to the $\CC$-bilinear extension of $g$.
Since $\nabla^{\Ch}\!I =0$, one has that $[\Lambda(Z), I] = 0$, hence $\Lambda(Z)(\frakm^{1,0}) \subset \frakm^{1,0}$ for every $Z\in\frakm^{\CC}$. Moreover, for every $Z \in \frakm^{1,0}$ and $W \in \frakm^{0,1}$ one has
\begin{equation}\label{E:LambdaBracket}
	\Lambda(Z)W = [Z,W]_{\frakm^{01}}, \quad
	\Lambda(W)Z = [W,Z]_{\frakm^{10}}.
\end{equation}
The proof of the last statements can be found in~\cite{Podesta}.

By the expression of the curvature of $\nabla^{\Ch}$ in~\eqref{E:tensors}, we have for every $V,W\in\frakm^{1,0}$,
\[
	\Ric_{\omega_I}(V,\bar{W})
		 = i \tr_{\frakm^{10}} ( \Rm_{\omega_I}(V,\bar{W}) )
		= 	- i \tr_{\frakm^{10}} (\Lambda([V,\bar{W}]_{\frakm^{\CC}})) 
			- i \tr_{\frakm^{10}} (\ad([V,\bar{W}]_{\frakl^{\CC}}))
\]
It is necessary to compute $\tr_{\frakm^{10}} \Lambda(Z)$ for a vector $Z\in\frakm^{\CC}$.
We first consider the trace of $\Lambda(Z_{\frakm^{01}})$.
We can write $Z_{\frakm^{01}} =  Z_{\frakt^{01}} + \sum_{\gamma\in\hat{R}^+} c_{\gamma}E_{-\gamma}$, for some complex numbers $c_{\gamma}$. 
Let $V\in \frakt^{1,0}$, then
\[
	\Lambda(Z_{\frakm^{01}})(V) 
	= [Z_{\frakm^{01}}, V]_{\frakm^{10}}
	= 0
\]
while for every $\alpha\in\hat{R}^+$
\[
	\Lambda(Z_{\frakm^{01}})(E_{\alpha}) 
		= [Z_{\frakm^{01}}, E_{\alpha}]_{\frakm^{10}}
		= 
		\alpha(Z_{\frakt^{01}})E_{\alpha}
		- c_{\alpha} \left(H_{\alpha}\right)_{\frakm^{10}}
		- \sum_{\substack{\gamma \colon
						\alpha-\gamma \in \hat{R}^+}} 
			c_{\gamma} N_{\alpha, -\gamma} E_{\alpha-\gamma}.
\]
Taking into account Lemmas~\ref{L:orthogonality} and~\ref{L:levels}, it follows that
\[
	\tr_{\frakm^{10}} \Lambda(Z_{\frakm^{01}}) = 
	\sum_{\alpha \in \hat{R}^+} \alpha( Z_{\frakt^{01}} ).
\]
Now $\tr_{\frakm^{10}} \Lambda(Z_{\frakm^{10}})$ is the opposite conjugate of $\tr_{\frakm^{10}} \Lambda(Z_{\frakm^{01}})$, 
hence
\[
\begin{aligned}
	\tr_{\frakm^{10}} \Lambda(Z)
		= 	\tr_{\frakm^{10}} \Lambda(Z_{\frakm^{01}}) + 
			\tr_{\frakm^{10}} \Lambda(Z_{\frakm^{10}}) 
		= \sum_{\alpha\in\hat{R}^+} \alpha( Z_{\frakt^{01}} ) - \overline{\alpha( Z_{\frakt^{01}} )}
			= \sum_{\alpha\in\hat{R}^+} \alpha( Z_{\frakt^{\CC}} ),
\end{aligned}
\]
where we used that $\overline{\alpha} = - \alpha$ for every root $\alpha$.
In a similar way, given $U\in \frakl^{\CC}$,  one can obtain
\[
	\tr_{\frakm^{10}} \ad( U ) 
		= \sum_{\alpha\in\hat{R}^+} \alpha( U_{\frakt_{\frakl}^{\CC}} ).
\]
We conclude that $\Ric_{\omega_I} (V, \bar{W}) = 0$ if $V,W\in\frakt^{1,0}$ and on $\frakn^{1,0}$ the only possibly non-zero components of the form are for $V = W = E_{\gamma}$, $\gamma\in\hat{R}^+$, namely
\[
\begin{aligned}
	\Ric_{\omega_I}(E_{\gamma}, \bar{E}_{\gamma})
	&=	i \tr_{\frakm^{10}} \Lambda( (H_{\gamma})_{\frakm^{\CC}} ) 
		+ \tr_{\frakm^{10}} \ad( (H_{\gamma})_{\frakl^{\CC}} )\\
	&=	i\sum_{\alpha\in\hat{R}^+} 	\alpha( (H_{\gamma})_{\frakt^{\CC}} ) +
									\alpha( (H_{\gamma})_{\frakt_{\frakl}^{\CC}} )\\
	&=	i\sum_{\alpha\in\hat{R}^+} 	\alpha( H_{\gamma} )
	=	2 i \gamma(H_{\hat{\delta}}),
\end{aligned}
\]
where
$\hat{\delta} := \frac{1}{2} \sum_{\alpha\in\hat{R}^+} \alpha$ and
$H_{\hat{\delta}} := \frac{1}{2} \sum_{\alpha\in\hat{R}^+} H_{\alpha}$.\\

We can now proceed with the proof of Theorem~\ref{T:HKT-E}.

\begin{proof}[Proof of Theorem~\ref{T:HKT-E}]
Let $(G/L, I, J)$ be a Joyce hypercomplex manifold endowed with an invariant HKT metric $g$.
We can now focus on the $J$-anti-invariant part of $\Ric_{\omega_I}$.

Consider first $H_j = H_{\alpha_j} - i I H_{\alpha_j}$. By \eqref{E:mainEquation}, $J\bar{H}_j = - \frac{1}{k_j}E_{\alpha_j}$, so
\begin{equation}\label{E:comparison1}
\begin{gathered}
\frac{\Ric_{\omega_I} - J\Ric_{\omega_I}}{2} 
	(H_j, \bar{H}_j) =
	\frac{1}{2\abs{k_j}^2} \Ric_{\omega_I}(E_{\alpha_j}, \bar{E}_{\alpha_j}) =
	\frac{i}{\abs{k_j}^2} \alpha_j ( H_{\hat{\delta}} )\\
\omega_I(H_j, \bar{H}_j) = \frac{i}{\abs{k_j}^2}g_j.
\end{gathered}
\end{equation}
%
%
%
Now let $\alpha_j$ be one of the strongly orthogonal roots chosen for the Joyce decomposition. Since $JE_{\alpha_j} \in \frakt^{0,1}$, we have $J\Ric_{\omega_I}(E_{\alpha_j}, \bar{E}_{\alpha_j}) = 0$. Therefore
\begin{equation}
\begin{gathered}
\frac{\Ric_{\omega_I} - J\Ric_{\omega_I}}{2} ( E_{\alpha_j}, \bar{E}_{\alpha_j} ) =
	i \alpha_j ( H_{\hat{\delta}} )\\
\omega_I(E_{\alpha_j}, \bar{E}_{\alpha_j}) = i g_j.
\end{gathered}
\end{equation}
%
%
%
Next we take $\gamma\in R_j^+$, $\gamma\ne\alpha_j$, for some $j\in\{1,\dots,m\}$. Then
\[
J\Ric_{\omega_I}(E_{\gamma}, \bar{E}_{\gamma}) = 4\abs{k_j}^2  (N_{\gamma,-\alpha_j})^2 2i(\gamma - \alpha_j) ( H_{\hat{\delta}} ),
\]
Since $(N_{\gamma,-\alpha_j})^2 = (\gamma, \alpha_j) = \frac{\abs{\alpha_j}^2}{2}$ (see Section~\ref{SS:homogeneous}), we have $4\abs{k_j}^2  (N_{\gamma,-\alpha_j})^2 = 1$,
so that $J\Ric_{\omega_I}(E_{\gamma}, \bar{E}_{\gamma}) = 2i(\gamma - \alpha_j) ( H_{\hat{\delta}} )$.
Therefore
\begin{equation}
\begin{gathered}
\frac{\Ric_{\omega_I} - J\Ric_{\omega_I}}{2} ( E_{\gamma}, \bar{E}_{\gamma} ) =
	i \alpha_j ( H_{\hat{\delta}} )\\
\omega_I(E_{\gamma}, \bar{E}_{\gamma}) = i g_j.
\end{gathered}
\end{equation}
%
%
%
Let now $\gamma, \sigma \in \hat{R}^+$ with $\gamma\ne\sigma$. We see that either $JE_{\gamma}$ is an element of $\frakt^{0,1}$ or it is a scalar multiple of $E_{\gamma - \alpha_j}$ for some $j$, and the same holds for $\sigma$. Thus
\begin{equation}
	\frac{\Ric_{\omega_I} - J\Ric_{\omega_I}}{2}(E_{\gamma}, \bar{E}_{\sigma}) = 0, \quad
	\omega_I(E_{\gamma}, \bar{E}_{\sigma}) = 0.
\end{equation}
%
%
%
With similar arguments, we get for $\gamma\in\hat{R}^+$ and $j\in\{1,\dots,m\}$
\begin{equation}
	\frac{\Ric_{\omega_I} - J\Ric_{\omega_I}}{2}(E_{\gamma}, \bar{H}_j) = 0, \quad
	\omega_I(E_{\gamma}, \bar{H}_j) = 0,
\end{equation}
and for every $j\ne l$
\begin{equation}\label{E:comparison2}
\frac{\Ric_{\omega_I} - J\Ric_{\omega_I}}{2} 
	(H_j, \bar{H}_l) = 0, \quad
	\omega_I(H_j, \bar{H}_l) = 0.
\end{equation}

Now, for every $j=1,\dots, m$ the constants
\[
	\alpha_j ( H_{\hat{\delta}} )
\]
are actually positive, see e.g. \cite[pag. 628]{PavingTheWay} . 
We conclude that $g$ is HKT-Einstein if and only if $g_j = \alpha_j ( H_{\hat{\delta}} )$ for every $j$, up to global scaling.
If moreover $L$ is connected, we can define an invariant HKT-Einstein metric on $G/L$ in this way.
\end{proof}

\begin{corollary}
Every simply connected Joyce hypercomplex manifold $(G/L, I, J)$ admits a unique (up to scaling), $G$-invariant HKT-Einstein metric.
\end{corollary}

\begin{remark}\label{R:comparison}
We have
\[
	\alpha_j(H_{\hat{\delta}}) = 
	\frac{1}{2}  \sum_{\gamma\in R_j^+} (\alpha_j, \gamma) +
	\frac{1}{2}  \sum_{k=1}^{j-1} \sum_{\gamma\in R_k^+} (\alpha_j, \gamma).
\]
The second sum vanishes, since for every $k=1,\dots,j-1$ and $\gamma\in R_k^+$ we can consider $\alpha_k - \gamma$, which is easily seen to be a root in $R_k^+$, and of course $(\alpha_j, \gamma) + (\alpha_j, \alpha_k - \gamma) = (\alpha_j, \gamma) + (\alpha_j, - \gamma) = 0$.
The first sum can be handled as follows. As noted in Section~\ref{SS:Joyce}, $2\frac{(\gamma, \alpha_j)}{(\alpha_j, \alpha_j)} = 1$ for every $\gamma\in R_j^+$,
hence
\[
\begin{aligned}
	\frac{1}{2} \sum_{\gamma\in R_j^+} (\alpha_j, \gamma) 
	&= \frac{(\alpha_j, \alpha_j)}{4}
		\left( 2 + \sum_{\gamma\in R_j^+} 2\frac{(\alpha_j, \gamma)}{(\alpha_j, \alpha_j)} \right)
	=\frac{(\alpha_j, \alpha_j)}{4} 
		\left( 2 + \dim_{\CC}\frakf_j \right)\\
	&=\frac{(\alpha_j, \alpha_j)}{4}
		\left( 2 + 2\dim_{\HH}\frakf_j \right)
\end{aligned}
\]
In the case of a Lie group (i.e., $L$ trivial), an invariant HKT-Einstein metric was already found in \cite{FusiGentili} by suitably modifying the negative of the Killing form.
This computation shows that our result agrees with the metric given there.
\end{remark}

\section{Properties of the Bismut connection}\label{S:Bismut}
Let $(G/L, I, g)$ be a homogeneous Hermitian manifold, where $G$ is a compact Lie group and $L$ is a closed connected subgroup. We assume that $g$ is projectable to the flag manifold $G/C$ (see Section~\ref{SS:homogeneous} for the notations).
Let us denote by $\nabla^{\Bis}$ the Bismut connection, and by $T^{\Bis}$ and $R^{\Bis}$ the torsion and curvature tensors; $\Lambda^{\Bis} \colon \frakm \to \End(\frakm)$ will denote the operator associated to $\nabla^{\Bis}$.

With the same arguments used in \cite[Sections 4, 5]{PodestaZheng}, one can prove that the only possibly non-zero values of $\Lambda^{\Bis}$ are:
\begin{equation}\label{E:LambdaBismut1}
	\Lambda^{\Bis}(E_{\alpha}) E_{\beta} = \frac{1}{2} N_{\alpha, \beta} 
		\left( 	1 + \epsilon_{\alpha}\epsilon_{\beta} 
				+ ( 1 - \epsilon_{\alpha}\epsilon_{\alpha + \beta}) \frac{g_{\beta}}{g_{\alpha+\beta}}
				- ( 1 + \epsilon_{\beta}\epsilon_{\alpha + \beta}) \frac{g_{\alpha}}{g_{\alpha+\beta}}
		\right) E_{\alpha + \beta},
\end{equation}
for every $\alpha, \beta \in \hat{R}$ such that $\alpha+\beta\in\hat{R}$, and
\begin{equation}\label{E:LambdaBismut2}
	\Lambda^{\Bis}(H) E_{\alpha} = 	\left( 	\frac{g(H, (H_{\alpha})_{\frakm^{\CC}} )}{g_{\alpha}} 
								+ \alpha(H)
								\right) E_{\alpha},
\end{equation}
for every $\alpha \in \hat{R}$ and $H\in \frakt^{\CC}$.
Here $\epsilon_{\alpha} = \pm 1$ if $\alpha\in\hat{R}^{\pm}$.

Furthermore, the following Lemma is still valid.
\begin{lemma}[\cite{PodestaZheng}, Lemma 4.3]\label{L:Bismut}
Let $\alpha,\beta\in\hat{R}^+$ be such that $\alpha + \beta \in \hat{R}^+$.
Assume that
\[
	\nabla^{\Bis}_{E_{-\alpha}} T^{\Bis} (E_{\alpha}, E_{\beta}) = 0,
	\qquad
	\nabla^{\Bis}_{E_{-\beta}} T^{\Bis} (E_{\alpha}, E_{\beta}) = 0.
\]
We have:
\begin{enumerate}
\item
if $\alpha - \beta \notin \hat{R}$, then	either $g_{\alpha+\beta} = g_{\alpha} + g_{\beta}$ or
$g_{\alpha+\beta} = g_{\alpha} = g_{\beta}$;
\item
if $\alpha - \beta \in \hat{R}$, then	either $g_{\alpha+\beta} = g_{\alpha} + g_{\beta}$, or
$g_{\alpha+\beta} = g_{\alpha}$, or
$g_{\alpha+\beta} = g_{\beta}$;
\end{enumerate}
\end{lemma}

With this preliminaries at hand, we now consider again a Joyce hypercomplex manifold $(G/L, I, J)$, endowed with a compatible invariant HKT metric $g$.
We start assuming $T^{\Bis}$ is parallel with respect to $\nabla^{\Bis}$.
We have seen that $\frakg$ admits a decomposition as in~\eqref{E:decl},\eqref{E:decl}, based on the choice of strongly orthogonal roots $\{\alpha_1, \dots, \alpha_d \}$.
By Theorem~\ref{T:HKT}, we can write
\begin{equation*}
	g = \sum_{j=1}^m g_j h|_{\frakm_j\times\frakm_j}.
\end{equation*}
It is immediate to verify the following statement.
\begin{lemma}\label{L:rootsum}
Let $\alpha\in R_j^+$, $\beta\in R_l^+$ for some $j,l \in \{1,\dots, m\}$, with $l\ge j$.
\begin{itemize}
\item
If $\alpha + \beta$ is a root, then $\alpha+\beta\in R_j^+$.
\item
If $\alpha - \beta$ is a root, then 
\[
\begin{array}{ll}
	\alpha - \beta \in R_j^+,										& \mbox{if $l > j$,}\\
	\alpha - \beta \in R_j \cup \cdots \cup R_m \cup R_{\frakl},	& \mbox{if $l = j$.}
\end{array}
\]
\end{itemize}
\end{lemma}

\begin{proposition}\label{P:BismutTorsion}
Let $(G/L, I, J)$ be a Joyce hypercomplex manifold. Then an invariant HKT metric $g$ on it has parallel Bismut torsion if and only if, with respect to the expression~\eqref{E:scalarproduct},
$g_j = g_l$ whenever $\alpha_j$, $\alpha_l$ belong to the same irreducible component of $R$.
In other words, if and only if $g|_{\frakn\times\frakn} = b|_{\frakn\times\frakn}$ for an $\Ad(G)$-invariant scalar product $b$ on $\frakg$.
\end{proposition}
\begin{proof}
Let $\Delta$ be one of the irreducible components of $R$ such that $\Delta\cap\hat{R} \neq \varnothing$.
Suppose $\alpha_{j_1}, \dots, \alpha_{j_p}$, $1 \leq j_1 < j_2 < \dots < j_p \leq m$, are the roots among $\{ \alpha_1, \dots, \alpha_m \}$ that lie in $\Delta$.
By construction $\alpha_{j_1}$ is the maximal root of $\Delta$.
A direct inspection on the Joyce decomposition of the nine types of simple Lie algebras, using the data in \cite[arXiv version, Sec. 3.3]{DimitrovTsanovI}, ensures that for every $r\in \{2,\dots,p\}$ there exist $\alpha \in R_{j_1}^+$ with the property that $\alpha+\alpha_{j_r}$ is a root and $\alpha - \alpha_{j_r}$ is not.
By Lemma~\ref{L:rootsum}, $\alpha+\alpha_{j_r} \in R_1^+$.
Now part (1) of Lemma~\ref{L:Bismut} implies that either $g_{j_1} = g_{j_1} + g_{j_r}$ or $g_{j_1} = g_{j_r}$.
The first case clearly is not possible, so we have $g_{j_1} = g_{j_r}$.
In the end, $\nabla^{\Bis}T^{\Bis} = 0$ implies the $g$-norms of $E_{\gamma}$ are all equal, for all roots $\gamma\in\hat{R}$ which belong to the same irreducible component of $R$.

If $\frakg = \frakz(\frakg) \oplus \frakg_1 \oplus\dots\oplus \frakg_t$ is the decomposition of $\frakg$ into the sum of center and simple ideals, then one can  define an $\Ad(G)$-invariant inner product $b = b_0 + \lambda_1(-\calB_{\frakg_1}) + \dots + \lambda_1(-\calB_{\frakg_t})$, such that $g|_{\frakn\times\frakn} = b|_{\frakn\times\frakn}$. 
Here $\frakn$ is as in Section~\ref{SS:homogeneous}.\\

We shall prove now that the property just stated is also sufficient to have $\nabla^{\Bis} T^{\Bis} = 0$.
Let $\alpha, \beta \in \hat{R}$ be such that $\alpha + \beta$ is a root and it still belongs to $\hat{R}$.
Then $\alpha$ and $\beta$ are in the same irreducible component of $R$ (otherwise $\alpha + \beta$ can not be a root), so $g_{\alpha} = g_{\beta}$ by our assumption.
By~\eqref{E:LambdaBismut1}, it follows that $\Lambda^{\Bis}(E_{\alpha})E_{\beta} = 0$, thus the only possibly non-vanishing component of the Bismut connection is given by~\eqref{E:LambdaBismut2}.
There are now three cases to examine.

We show first that $\nabla^{\Bis}_{H}T^{\Bis}(E_{\alpha}, E_{\beta}) = 0$ for every $H\in \frakt^{\CC}$ and $\alpha,\beta\in\hat{R}$ with $\alpha \ne -\beta$.
The only non-trivial case is when $\alpha$ and $\beta$ are in the same irreducible component and $\alpha + \beta \in \hat{R}$. 
\[
\begin{aligned}
	&\Lambda^{\Bis}(H)T^{\Bis} (E_{\alpha}, E_{\beta}) \\
	&= 	\Lambda^{\Bis}(H) \left( T^{\Bis} (E_{\alpha}, E_{\beta}) \right) 
			-  	T^{\Bis} (\Lambda^{\Bis}(H)E_{\alpha}, E_{\beta}) 	- 	T^{\Bis} (E_{\alpha}, \Lambda^{\Bis}(H)E_{\beta}) \\
	&= - 	\Lambda^{\Bis}(H) [E_{\alpha}, E_{\beta}] 
		 -  	\left( 	
						\frac{g(H, (H_{\alpha})_{\frakm^{\CC}} )}{g_{\alpha}} + \alpha(H) + \frac{g(H, (H_{\beta})_{\frakm^{\CC}} )}{g_{\beta}} + \beta(H)
			\right) T^{\Bis}(E_{\alpha}, E_{\beta})\\
	&= 		- N_{\alpha, \beta}\left( 
									\frac{ g( H, (H_{\alpha+\beta})_{\frakm^{\CC}} ) }{g_{\alpha+\beta}} + (\alpha + \beta)(H) 
							\right) E_{\alpha+\beta}\\
	&\quad 	+ N_{\alpha, \beta}  	\left( 
										\frac{g(H, (H_{\alpha})_{\frakm^{\CC}} )}{g_{\alpha}} + \alpha(H)  + \frac{g(H, (H_{\beta})_{\frakm^{\CC}} )}{g_{\beta}} + \beta(H) 
								\right) E_{\alpha+\beta},
\end{aligned}
\]
and this expression vanishes since by assumption $g_{\alpha} = g_{\beta} = g_{\alpha+\beta}$.

Next we check that $\nabla_H^{\Bis}T^{\Bis} (E_{\alpha}, E_{-\alpha}) = 0$ for every $\alpha\in\hat{R}$ and $H\in\frakt^{\CC}$:
\[
\begin{aligned}
	&\Lambda^{\Bis}(H)T^{\Bis}(E_{\alpha}, E_{-\alpha})\\
	&=		\Lambda^{\Bis}(H) \left(T^{\Bis} \left( E_{\alpha}, E_{-\alpha} \right)\right) - T^{\Bis}( \Lambda^{\Bis}(H)E	_{\alpha}, E_{-\alpha} ) - T^{\Bis}(E	_{\alpha}, \Lambda^{\Bis}(H)E_{-\alpha} )\\
	&=		-\Lambda^{\Bis}(H)[E_{\alpha}, E_{-\alpha}]_{\frakm^{\CC}}\\
	&\quad	- \left(
					 				\frac{g(H, (H_{\alpha})_{\frakm^{\CC}} )}{g_{\alpha}} + \alpha(H)  + \frac{g(H, (H_{-\alpha})_{\frakm^{\CC}} )}{g_{\alpha}} - \alpha(H)
				\right)  [E_{\alpha}, E_{-\alpha}]_{\frakm^{\CC}} = 0.
\end{aligned}
\]
The last case is $\nabla^{\Bis}_H T^{\Bis}(V, E_{\alpha}) = 0$, for $\alpha\in\hat{R}$ and $H,V\in\frakt^{\CC}$, which is trivially verified.
\end{proof}

We have that, if the Bismut torsion $T^{\Bis}$ is Bismut-parallel, also the Bismut curvature $R^{\Bis}$ has this property. This behavior was observed in~\cite[Remark 5.2]{PodestaZheng} for compact semisimple Lie groups with a Samelson complex structure and several other cases.

\begin{proposition}\label{P:BismutCurvature}
Let $(G/L, I, J)$ be a Joyce hypercomplex manifold. If an invariant HKT metric $g$ on it has $\nabla^{\Bis}$-parallel Bismut torsion, then it has $\nabla^{\Bis}$-parallel Bismut curvature.
\end{proposition}
\begin{proof}
Under the hypothesis $\nabla^{\Bis}T^{\Bis} = 0$, the only non-zero component of $\Lambda^{\Bis}$ is given by \eqref{E:LambdaBismut2}, therefore those of $R^{\Bis}$ are
\[
\begin{gathered}
	R^{\Bis}(E_{\alpha}, E_{\beta})Z = - N_{\alpha,\beta} \ad(E_{\alpha+\beta})Z
	 \qquad \text{ if $\alpha+\beta \in R_{\frakl}$, and $Z\in\frakm^{\CC}$ }\\	
	R^{\Bis}(E_{\alpha}, E_{-\alpha}) E_{\gamma} 
		= - \Lambda^{\Bis}( (H_{\alpha})_{\frakm} ) E_{\gamma} 
			- \ad( (H_{\alpha})_{\frakl} ) E_{\gamma}
		= - 	\left(
				\frac{g( (H_{\gamma})_{\frakm}, (H_{\alpha})_{\frakm} )}{g_{\gamma}} 
				+ \gamma( H_{\alpha} )
			\right)	E_{\gamma}.
\end{gathered}			
\]

It follows that the components of $\nabla^{\Bis}R^{\Bis}$ which may be non-zero are 
$\nabla^{\Bis}_{H}R^{\Bis}(E_{\alpha}, E_{\beta}) E_{\gamma}$,
for $H\in \frakt^{\CC}$, $\alpha,\beta,\gamma \in \hat{R}$.
Let us assume first $\alpha + \beta \ne 0$. Then we have to consider only the case in which the semisimple part of $\frakl$ is non-trivial and $\alpha+\beta\in R_{\frakl}$:
\[
\begin{aligned}
	\nabla^{\Bis}_{H}R^{\Bis} (E_{\alpha}, E_{\beta}) V
	&= \Lambda^{\Bis}(H) \left(R^{\Bis} (E_{\alpha}, E_{\beta}) V \right)\\
	&\quad 
			- R^{\Bis}( \Lambda^{\Bis}(H) E_{\alpha}, E_{\beta}) V
			- R^{\Bis}(  E_{\alpha}, \Lambda^{\Bis}(H) E_{\beta}) V\\
	&\quad 	- R^{\Bis}(E_{\alpha}, E_{\beta}) ( \Lambda^{\Bis}(H)V )  \\
	&=	 	- N_{\alpha,\beta} \Lambda^{\Bis}(H) [ E_{\alpha+\beta}, V ] \\
	&\quad	+ \left(
				\frac{g( H, (H_{\alpha})_{\frakm} )}{g_{\alpha}} + \alpha(H)
				+ \frac{g( H, (H_{\beta})_{\frakm} )}{g_{\beta}} + \beta(H)
			\right)
			N_{\alpha,\beta} [E_{\alpha+\beta}, V] \\
	&\quad 	+  N_{\alpha,\beta} [ E_{\alpha+\beta}, \Lambda^{\Bis}(H)V ].
\end{aligned}
\]			
If $V\in\frakt^{\CC}$, each summand vanishes (note $[E_{\alpha+\beta}, V]=0$ since $[\frakl, \frakt]=0$).
If $V=E_{\gamma}$ for some $\gamma$ in $\hat{R}$, then it is enough to consider
$\alpha+\beta+\gamma \in \hat{R}$, so that
\begin{multline*}
	\nabla^{\Bis}_{H}R^{\Bis} (E_{\alpha}, E_{\beta}) E_{\gamma} =
	N_{\alpha,\beta}N_{\alpha+\beta, \gamma}
	\Big(		- \frac{g(H, (H_{\alpha+\beta+\gamma})_{\frakm})}{g_{\alpha+\beta+\gamma}}
			- (\alpha+\beta+\gamma)(H) \\
			+ \frac{g(H, (H_{\alpha})_{\frakm})}{g_{\alpha}}
			+ \alpha(H)
			+ \frac{g(H, (H_{\beta})_{\frakm})}{g_{\beta}}
			+ \beta(H)
			+ \frac{g(H, (H_{\gamma})_{\frakm})}{g_{\gamma}}
			+ \gamma(H)
	\Big) E_{\alpha+\beta+\gamma}.
\end{multline*}
Since necessarily $\alpha, \beta, \gamma, \alpha+\beta+\gamma$ are in the same irreducible component,
we have $g_{\alpha+\beta+\gamma}= g_{\alpha} = g_{\beta} = g_{\gamma}$, hence the sum vanishes.

Assume instead $\alpha = -\beta$. We compute, as above, that
\begin{multline*}
	R^{\Bis}( \Lambda^{\Bis}(H) E_{\alpha}, E_{-\alpha}) E_{\gamma} +
	R^{\Bis}(  E_{\alpha}, \Lambda^{\Bis}(H) E_{-\alpha}) E_{\gamma} = \\
=	\left(
		\frac{g( H, (H_{\alpha})_{\frakm} )}{g_{\alpha}} + \alpha(H)
		+ \frac{g( H, (H_{-\alpha})_{\frakm} )}{g_{\alpha}} - \alpha(H)
	\right)
	R^{\Bis}(E_{\alpha}, E_{-\alpha}) E_{\gamma} = 0
\end{multline*}
and
\[
\begin{aligned}
	&\Lambda^{\Bis}(H)
	\left(
		R^{\Bis}(E_{\alpha}, E_{-\alpha}) E_{\gamma}
	\right) =\\ 
	&\quad=
	\left(
		\frac{g( (H_{\gamma})_{\frakm}, (H_{\alpha})_{\frakm} )}{g_{\gamma}} 
		+ \gamma( H_{\alpha} )E_{\gamma}
	\right)\Lambda^{\Bis}(H)E_{\gamma} \\
	&\quad= 
	\left(
		\frac{g( (H_{\gamma})_{\frakm}, H_{\alpha} )}{g_{\gamma}} 
		+ \gamma( (H_{\alpha})_{\frakm} )E_{\gamma}
	\right)
	\left(
		\frac{g( (H_{\gamma})_{\frakm}, H)}{g_{\gamma}} 
		+ \gamma(H) E_{\gamma}
	\right) E_{\gamma} \\
	&\quad = R^{\Bis}(E_{\alpha}, E_{-\alpha}) ( \Lambda^{\Bis}(H)E_{\gamma} ).
\end{aligned}
\]
Thus 
\[
\begin{aligned}
	\nabla^{\Bis}_{H}R^{\Bis} (E_{\alpha}, E_{-\alpha}) E_{\gamma}
	&= \Lambda^{\Bis}(H) \left(R^{\Bis} (E_{\alpha}, E_{-\alpha}) E_{\gamma} \right)\\
	&\quad 
			- R^{\Bis}( \Lambda^{\Bis}(H) E_{\alpha}, E_{-\alpha}) E_{\gamma} 
			- R^{\Bis}(  E_{\alpha}, \Lambda^{\Bis}(H) E_{-\alpha}) E_{\gamma}\\
	&\quad 	- 	R^{\Bis}(E_{\alpha}, E_{-\alpha}) ( \Lambda^{\Bis}(H)E_{\gamma} ) \\
	&= 0	.
\end{aligned}
\]
\end{proof}

\begin{remark}
Given a homogeneous Hermitian manifold $(G/L, I, g)$, were $G$ is a compact Lie group and $L$ a closed subgroup, with the same computations done in the last two Propositions one can show that if $g|_{\frakn\times\frakn} = b|_{\frakn\times\frakn}$, for an $\Ad(G)$-invariant scalar product $b$ on $\frakg$, then the torsion and the curvature of the Bismut connection are parallel.
\end{remark}

Finally we make some observations on invariant \emph{strong} HKT metrics, completing the proof of Theorem~\ref{T:Bismut}.

\begin{proposition}\label{P:SHKT}
Let $(M=G/L, I, J)$ be a Joyce hypercomplex manifold with an invariant strong HKT metric $g$. 
The following facts hold.
\begin{enumerate}
\item[(i)]
The torsion and the curvature of the Bismut connection are $\nabla^{\Bis}$-parallel.
\item[(ii)]
Assume $M = \U(1)^{\ell} \times (K_1/L_1) \times \cdots \times (K_t/L_t)$, where $K_j$ are compact, connected, simple Lie groups, and $L_j\subset K_j$ are closed, connected subgroups. Then $M$ is a Lie group and the HKT structure  $(I,J,g)$ is left-invariant.
\end{enumerate}
\end{proposition}
\begin{proof}
As in section~\ref{SS:Joyce}, associated to the homogeneous space $G/L$ we have a decomposition given by~\eqref{E:decl} and in~\eqref{E:decm}. In particular, if $d$ denotes the number of levels in the full Joyce decomposition of $\frakg$, for an integer $1\leq m \leq d$ we have $\hat{R} = \bigcup_{j=1}^m R_j$, $R_{\frakl} = \bigcup_{j=m+1}^d R_j$.
Let us fix any index $k\in \{1,\dots, d \}$.
Let $\Delta \subset R$ be the irreducible component of $R$ containing $\alpha_k$, and let $\mu$ be the (unique) maximal root of $\Delta$. 
By construction, $\mu = \alpha_j \in \hat{R}$ for some $j \leq k$. If $j < k$,
a direct inspection ensures that 
there exist $\theta,\beta \in R_j^+$ such that
$\theta - \beta = \alpha_k \in R_k^+$, and $\alpha_j - (\theta + \beta)$ is not a root.
Then let us set
\[
	\alpha = \alpha_j - \beta \in R_j^+, \qquad
	\gamma= - \theta \in R_j^-, \qquad
	\nu= \theta - \alpha_j \in R_j^-.
\]
Clearly $\alpha + \beta + \gamma + \nu = 0$ and all pairs add up non-zero.
Moreover we have
\[
	\alpha + \beta = \alpha_j \in R_j^+, \qquad
	\alpha + \nu = -\beta + \theta = \alpha_k \in R_k^+,
\]
while $\alpha + \gamma = \alpha_j - (\beta + \theta)$ is not a root.
For example, if $\Delta$ is of type $\operatorname{A}_n$, for $n\ge 3$ (for $n=1,2$ there is only one level of decomposition), in the notations of Example~\ref{Example} we have $\mu = \vareps_1 - \vareps_2$, and for every 
$2 \leq p \leq \left[\frac{n+1}{2}\right]$ we can take
$\theta = \vareps_1 - \vareps_{2p}$, $\beta = \vareps_1 - \vareps_{2p -1}$, so that
\[
	\alpha = \vareps_{2p-1} - \vareps_{2}, , 
	\qquad
	\gamma = \vareps_{2p} - \vareps_{1}, 
	\qquad
	\nu = \vareps_{2} - \vareps_{2p}.
\]

In order to prove claim (i), we consider the case $1 \leq k \leq m$. Assume $j < k$. With computations similar to those in \cite[Lemma 3.1]{LauretMontedoro}, one obtains
\[
\begin{aligned}
	0 
	&=\dd{\dd{}^{\cplx} \omega_I} (E_{\alpha}, E_{\beta}, E_{\gamma}, E_{\nu})\\
	& = N_{\alpha,\beta} N_{\gamma, \nu}
		(g_{\alpha} + g_{\beta} + g_{\gamma} + g_{\nu} - 2 g_{\alpha+\beta} ) \\
	&\quad	+ \epsilon_{\alpha+\nu} N_{\alpha,\nu} N_{\beta, \gamma}
		(- g_{\alpha} + g_{\beta} - g_{\gamma} + g_{\nu} 
			+ 2 \epsilon_{\alpha+\nu} g_{\alpha+\nu} ) \\
	&=   N_{\alpha,\beta} N_{\gamma, \nu}
			( g_j + g_j + g_j + g_j - 2 g_j )
		 + \epsilon_{\alpha+\nu} N_{\alpha,\nu} N_{\beta, \gamma}
			(- g_j + g_j - g_j + g_j 
			+ 2 g_k ) \\
	&=	2N_{\alpha,\beta} N_{\gamma, \nu} g_j  + 2N_{\alpha,\nu} N_{\beta, \gamma} g_k\\
	&= 2N_{\alpha,\beta} N_{\gamma, \nu} ( - g_j + g_k ).
\end{aligned}
\]
Here we used the formula
$N_{\alpha,\beta} N_{\gamma, \nu} + N_{\alpha,\nu} N_{\beta, \gamma} = 0$.
It follows that $g_j = g_k$. Since $k$ was chosen in $\{ 1, \dots, m\}$, the proof of the first claim is finished applying Propositions~\ref{P:BismutTorsion} and~\ref{P:BismutCurvature}.

To prove claim (ii), we consider the case $m+1 \leq k \leq d$.
Suppose the maximal root $\alpha_j$ is in $\hat{R}$, then $j \leq m < k$.
Since $\alpha + \nu = \alpha_k \in R_{\frakl}$, this time one obtains
\[
		0 
	=\dd{\dd{}^{\cplx} \omega_I} (E_{\alpha}, E_{\beta}, E_{\gamma}, E_{\nu})\\
	= N_{\alpha,\beta} N_{\gamma, \nu}
		(g_{\alpha} + g_{\beta} + g_{\gamma} + g_{\nu} - 2 g_{\alpha+\beta} )
	= 	2 N_{\alpha,\beta} N_{\gamma, \nu} g_j,
\]
a contradiction. Thus $\alpha_j \in R_{\frakl}$, and hence $\Delta \subset R_{\frakl}$.

By the current assumption we have $\frakg = \RR^{\ell}\oplus \frakk_1\oplus\dots\oplus\frakk_{t}$, $\frakl=\frakl_1\oplus\cdots\oplus\frakl_t$, and the chosen maximal abelian subalgebra has the form $\RR^{\ell}\oplus\fraks_1\oplus\cdots\fraks_t$, each $\fraks_j$ being a maximal abelian subalgebra of $\frakk_j$.
The irreducible components of $R$ are the root systems $\Delta_j$ associated to the simple Lie algebras $\frakk_j$.
The above argument shows that, for some $q$, $\hat{R} = \Delta_1 \cup \cdots \cup \Delta_q$, $R_{\frakl} = \Delta_{q+1} \cup \cdots \cup \Delta_t$, up to renaming indexes.
Thus  $\frakk_j =\frakl_j$ for every $j\in\{q+1, \dots, t\}$, hence the coset spaces $K_{q+1}/L_{q+1}, \dots, K_t/L_t$ are trivial.
Without loss of generality, we can therefore assume $R_{\frakl}=\varnothing$ and $M=\U(1)^{\ell} \times (K_1/L_1) \times \cdots \times (K_q/L_q)$.
According to~\eqref{E:decl}, $\frakl = \frakl_1\oplus\cdots\oplus\frakl_q$ is contained in the chosen maximal abelian subalgebra of $\frakg$, so $\frakl_j\subset\fraks_j$ for each $j\in\{1,\dots,q\}$.
Using again the pluriclosed condition, for every $\alpha,\beta \in\Delta_j^+$, $\alpha\neq\beta$, one has (\cite{LauretMontedoro})
\[
\begin{aligned}
	0	&=	\dd{\dd{}^{\cplx} \omega_I} (E_{\alpha},  E_{-\alpha}, E_{\beta}, E_{-\beta})\\
		&=	2g( (H_{\alpha})_{\frakm^{\CC}}, (H_{\beta})_{\frakm^{\CC}} )
			- 2 N_{\alpha,\beta}^2 (g_{\alpha+\beta} - g_{\alpha} - g_{\beta})
			- 2 \epsilon_{\alpha-\beta} N_{\alpha,-\beta}^2 
				(\epsilon_{\alpha-\beta}g_{\alpha-\beta} - g_{\alpha} + g_{\beta}).
\end{aligned}	
\]
By part (i), there is $\lambda>0$ such that $g_{\gamma} = \lambda$ for every $\gamma\in\Delta_j$, so
\[
	g(  (H_{\alpha})_{\frakm^{\CC}}, (H_{\beta})_{\frakm^{\CC}}  ) = 
	\lambda(- N_{\alpha,\beta}^2 + N_{\alpha,-\beta}^2  ) =
	\lambda (\alpha,\beta) = \lambda \calB( H_{\alpha} , H_{\beta} ),
\]
for every $\alpha, \beta \in \Delta_j^+$, $\alpha\neq\beta$. 
By construction $g(H_{\alpha_j}, H_{\alpha_j})= \lambda\calB(H_{\alpha_j}, H_{\alpha_j})$.
For every $\beta\in\Delta_j^+$, $\beta\neq\alpha_j$, we have
$g((H_{\beta})_{\frakm^{\CC}}, (H_{\alpha_j-\beta})_{\frakm^{\CC}}) = \lambda\calB(H_{\beta}, H_{\alpha_j-\beta})$,
hence $g((H_{\beta})_{\frakm^{\CC}}, (H_{\beta})_{\frakm^{\CC}}) = \lambda\calB(H_{\beta}, H_{\beta})$.
From this, it easily follows that $\lambda\calB(Z,Z) = g( (Z)_{\frakm^{\CC}}, (Z)_{\frakm^{\CC}}  )$ for every $Z \in \fraks_j^{\CC}$.
In particular, for $X\in\frakl_j$ we have $\lambda\calB(X,X) = g( (X)_{\frakm^{\CC}}, (X)_{\frakm^{\CC}}  ) = 0$, hence $\frakl_j = 0$.
Therefore the subgroups $L_1, \dots, L_q$ are trivial, and claim (ii) now follows.

\end{proof}

Putting together part (ii) of the last Proposition, and by~\cite[Corollary 5.12]{DimitrovTsanovII}, we have the following
\begin{corollary}
If a \emph{simply connected} Joyce hypercomplex manifold $(M,I,J)$ admits an invariant strong HKT metric, then $M = \SU(2n_1+1) \times \cdots \times \SU(2n_r+1)$.
\end{corollary}

\begin{remark}
Thanks to~\cite[Corollary 5.1]{FinoGrantcharov} and~\cite[Proposition 4.1]{LauretMontedoro}, if the Joyce hypercomplex manifold is a compact semisimple Lie group $G$, a compatible left-invariant HKT metric is strong if and only if it is bi-invariant.

If $G$ is not semisimple, we have seen in Section~\ref{S:HKT} that, for a given invariant hypercomplex structure, the additional hypothesis $\frakz(\frakg) = \frakb_d$ implies that the HKT metric $h$ as in~\eqref{E:h} is also bi-invariant, hence strong. 
Conversely, any other invariant strong HKT metric, compatible with the same structure, must be bi-invariant by Theorem~\ref{T:HKT} and Proposition~\ref{P:SHKT}.

However, if the additional assumption is dropped, this does not happen always.
An example of a strong HKT metric which is not bi-invariant can be given on $G = \U(2n)$.
Here the Joyce decomposition is
\[
	\frakg = \frakz(\frakg) \oplus \frakg_{ss},
	\qquad
	\frakz(\frakg) = \fraku(1),
	\qquad
	\frakg_{ss} = \frsu(2n)	=	
			\frakb'_{n}
			\oplus
			\bigoplus_{j=1}^{n} \frakd_j   	
			\oplus  	
			\bigoplus_{j=1}^{n-1} \frakf_j,
\]
where the Joyce decomposition of $\frsu(2n)$ is described in Example~\ref{Example} (note $\frakf_{n} = 0$).
Let $\{ Y^1, \dots, Y^{n-1} \}$ be a $(-\calB)$-orthonormal basis of $\frakb'_n$, and let $Y^{n}$ be a generator of $\fraku(1)$.
Consider the hypercomplex structure associated to the following basis of $\frakb_n = \fraku(1)\oplus\frakb_n'$
\[
	X_1^j = \frac{2}{\abs{\alpha_j}} Y^j
	\quad
	\text{ for $j=1,\dots,n-1$,}
	\quad\quad
	X_1^{n} = \frac{2}{\abs{\alpha_j}} (Y^1 + \dots + Y^{n-1}).
\]
Then the HKT metric $h$ in~\eqref{E:h} is not bi-invariant, as $\frakz(\frakg)$ and $\frakg_{ss}$ are not $h$-orthogonal. 
However it is strong, a fact that can be shown with a direct computation noting that $h|_{\frakg_{ss} \times \frakg_{ss}} = - \calB$.
\end{remark}

\begin{remark}
A consequence of Proposition~\ref{P:SHKT} is that, when the Joyce decomposition associated to $G/L$ features more than one level, an invariant HKT-Einstein metric is very often \emph{not strong}.
Consider for instance a homogeneous space of the form $M = \U(1)^{\ell}~\times~(K/L)$, where $K$ is a compact, connected, simple Lie group. 
The unique (up to scaling) HKT-Einstein metric corresponds to the coefficients
\[
	g_j = \alpha_j(H_{\hat{\delta}}) = 
	\frac{(\alpha_j, \alpha_j)}{4} \left( 2 + \dim_{\CC}\frakf_j \right),
	\quad j = 1, \dots, m,
\]
by Remark~\ref{R:comparison}, where $m$ is the number of layers.
One can easily compute these coefficients, verifying that they are not all equal, hence an invariant HKT-Einstein metric on $M$ can not be strong, unless $m=1$.
\end{remark}


\end{document}